\documentclass[12pt,reqno]{amsart}
\usepackage{multicol,amsmath,graphics,cancel,soul,color,bbm,amsfonts,dsfont,tikz,mathrsfs}
\usepackage[left=1in, right=1in, top=1.1in,bottom=1.1in]{geometry}
\usetikzlibrary{matrix,patterns,positioning}
\numberwithin{equation}{section}

\newcommand{\Cc}{\mathcal{C}}

\newcommand{\Fc}{\mathcal{F}}

\newcommand{\C}{\mathbb{C}}
\newcommand{\D}{\mathbb{D}}
\newcommand{\E}{\mathbb{E}}

\newcommand{\N}{\mathbb{N}}
\newcommand{\Pb}{\mathbb{P}}
\newcommand{\Q}{\mathbb{Q}}
\newcommand{\R}{\mathbb{R}}

\newcommand{\Hg}{\mathfrak{H}}
\newcommand{\Sf}{\mathscr{S}}
\newcommand{\Ef}{\mathscr{E}}
\newcommand{\ud}{\mathrm{d}}

\newcommand{\Norm}[1]{\left\lVert#1\right\rVert}
\newcommand{\Abs}[1]{\left|#1\right|}
\newcommand{\Ip}[1]{\left\langle #1 \right\rangle}
\newcommand{\Ipn}[1]{\langle #1 \rangle}
\newcommand{\Ipm}[1]{\big\langle #1 \big\rangle}
\newcommand{\Indi}[1]{\mathbbm{1}_{#1}}

\newcounter{dummy} \numberwithin{dummy}{section}
\newtheorem{Definition}[dummy]{Definition}
\newtheorem{Proposition}[dummy]{Proposition}
\newtheorem{Theorem}[dummy]{Theorem}

\newtheorem{Lemma}[dummy]{Lemma}

\def\1{{\rm l}\hskip -0.21truecm 1}

\begin{document}
\title[]{Convergence of the empirical spectral distribution of Gaussian matrix-valued processes}
 \date{\today}

\author[]{Arturo Jaramillo, Juan Carlos Pardo \and Jos\'e Luis P\'erez}
\address{Arturo Jaramillo: Department of Mathematics, University of Kansas,   Lawrence, KS 66045, USA.}
\email{jagil@ku.edu}
\address{Juan Carlos Pardo: Centro de Investigaci\'on en Matem\'aticas A.C., Calle Jalisco s/n, CP 36240, Guanajuato, Mexico.}
\email{jcpardo@cimat.mx}

\address{Jos\'e Luis P\'erez: Centro de Investigaci\'on en Matem\'aticas A.C., Calle Jalisco s/n, CP 36240, Guanajuato, Mexico.}
\email{jluis.garmendia@cimat.mx}
\date{\today}
 
\begin{abstract}
For a given normalized Gaussian symmetric matrix-valued process $Y^{(n)}$, we consider the process of its eigenvalues
$\{(\lambda_{1}^{(n)}(t),\dots, \lambda_{n}^{(n)}(t)); t\ge 0\}$  as well as its corresponding process of empirical spectral measures $\mu^{(n)}=(\mu_{t}^{(n)}; t\geq0)$.
Under some mild conditions on the covariance function associated to $Y^{(n)}$, we prove that the process $\mu^{(n)}$ converges in probability to a deterministic limit $\mu$, in the topology of uniform convergence over compact sets. We show that the process $\mu$ is characterized by its Cauchy transform, which is a rescaling of the solution of a Burgers' equation. Our results extend those of  Rogers and Shi \cite{RoSh} for the free Brownian motion  and Pardo et al. \cite{PaGaPe} for the non-commutative fractional Brownian motion when $H>1/2$    whose  arguments use strongly the non-collision of the eigenvalues. Our methodology does not require the latter property and in particular explains the remaining case of the non-commutative fractional Brownian motion for $H< 1/2$ which, up to our knowledge,  was unknown. 
\end{abstract}

\maketitle

\section{Introduction}

Let us consider a family of independent centered Gaussian processes $\{X_{i,j}; i,j\in\N\}$ defined in a probability space $(\Omega, \Fc, \Pb)$, with common covariance function here denoted by $R(s,t)$, for $s,t\ge 0$. That is to say, the Gaussian processes  $X_{i,j}:=(X_{i,j}(t); t\geq0)$ are independent  with zero mean and covariance given by 
\begin{align*}
\E\left[X_{i,j}(s)X_{i,j}(t)\right]
  &=R(s,t),\qquad \textrm{for} \qquad s, t\ge 0, 
\end{align*}
where $R(s,t)$ is a  non-negative definite covariance function.  For $n\in \mathbb{N}$, we also consider the renormalized symmetric Gaussian matrix-valued  process  $Y^{(n)}(t):=[Y_{i,j}^{(n)}(t)]_{1\leq i,j\leq n}$, for $t\ge 0$,  defined as follows
$$Y_{i,j}^{(n)}(t)
  :=\left\{\begin{array}{cc}
	\displaystyle\frac{1}{\sqrt{n}}X_{i,j}(t) +A_{i,j}^{(n)}&\ \ \text{ if }\ i<j,\\
	\displaystyle\frac{\sqrt{2}}{\sqrt{n}}X_{i,i}(t) + A_{i,i}^{(n)}&\ \ \text{ if }\ i=j,
	\end{array}\right.$$ 
where the $A_{i,j}^{(n)}$ are the coefficients of a deterministic symmetric  matrix $A^{n}=[A_{i,j}^n]_{1\leq i,j\leq n}$. Let us denote the  $n$-dimensional process of eigenvalues of $Y^{(n)}$ by  
$(\lambda_{1}^{(n)}(t), \cdots, \lambda_{n}^{(n)}(t))$, for $t\ge 0$. We also denote by $\mathtt{Pr}(\R)$ for the space of probability measures on $\mathbb{R}$ endowed with the topology of weak convergence and let $\mathcal{C}(\mathbb{R}_+, \mathtt{Pr}(\R))$ be the space of continuous functions from $\mathbb{R}_+$ into $\mathtt{Pr}(\R)$, endowed with the topology of uniform convergence on compact intervals of $\mathbb{R}_+$. 

In this manuscript, we are interested in the asymptotic behaviour of the $\mathtt{Pr}(\R)$-valued process of empirical distributions $\{\mu^{(n)}; n\geq1\}$, defined by $\mu^{(n)}:=(\mu_{t}^{(n)}, t\geq0)$ where
\begin{align*}
\mu_{t}^{(n)}
  &:=\frac{1}{n}\sum_{j=1}^{n}\delta_{\lambda_{j}^{(n)}(t)},\ \ \ \ t\geq0,
\end{align*}
and $\delta_{x}$ denotes the Dirac measure centered at $x$. In particular, we aim to determine the limit in probability of the process $\mu^{(n)}$, viewed as an element of the space $\Cc(\mathbb{R}_+, \mathtt{Pr}(\R))$ of $\mathtt{Pr}(\R)$-valued stochastic processes with continuous trajectories.

This problem has been  studied before in the framework of interacting particles by Rogers and Shi \cite{RoSh} and  C\'epa and L\'epingle in \cite{CeLe}, when  the $X_{i,j}$'s are standard Brownian motions. We also refer to Cabanal-Duvillard  and Guionnet \cite{CaGu} for the case of Hermitian Brownian motion where the latter case is included.
The authors in \cite{CeLe, RoSh} proved that $\mu^{(n)}$ converges, as $n$ tends to infinity, to a deterministic process whose Cauchy transforms are given by the solution of a Burgers' equation. More recently, Pardo et al. \cite{PaGaPe}  extended the previous result to the case where the $X_{i,j}$'s are fractional Brownian motions with Hurst (or self-similar) index $H>1/2$. We briefly describe the main ideas presented in all these manuscripts where the no-collision of the eigenvalues is crucial. Let $\Cc^{r}(\R)$ denote the set of real-valued functions with continuous derivatives of order $r$, and let us introduce  the subset
\begin{align}\label{eq:Cb2R}
\mathcal{C}_{b}^{r}(\R):=\bigg\{f\in \Cc^{r}(\R)\bigg|\ \sum_{i=1}^{r}\sup_{x\in\R}\Abs{f^{(r)}(x)}<\infty\bigg\}.
\end{align}
In the  Brownian case, the main idea for determining the asymptotic behaviour of $\mu^{(n)}$ consists, first, in characterizing the process of its eigenvalues $(\lambda_{1}^{(n)},\dots, \lambda_{n}^{(n)})$ as the unique strong solution of a system of stochastic differential equations. Then one can prove that for every $f\in \Cc_{b}^3(\R)$, the process 
\begin{equation}\label{aux_1}
\langle\mu^{(n)}_t,f\rangle:=\int f(x)\mu^{(n)}_t(\ud x), \qquad t\ge 0,
\end{equation}
 satisfies a stochastic differential equation  which converges, as $n$ tends to infinity, to a deterministic differential equation with a given initial condition. After a suitable approximation argument, one can also prove that the Cauchy transform of $\mu_{t}^{(n)}$,  defined by
\[
G^{(n)}(z)=\int (x-z)^{-1}\mu_{t}^{(n)}(\ud x) \qquad \textrm{for} \quad z\in\mathbb{C}_+, 
\] 
converges to the unique solution of a deterministic Burgers' equation. A key ingredient in this argument consists on using the well known fact that for any fixed $n\in\N$, the eigenvalues $(\lambda_{1}^{(n)},\dots,\lambda_{n}^{(n)})$ never collide, in other words, the trajectories of $\lambda_{i}^{(n)}$ and $\lambda_{j}^{(n)}$ never intersect for any $1\leq i,j\leq n$, and satisfy the following  non-colliding diffusion equation 
\begin{align}\label{eq:lambdadyson}
\lambda_{i}^{(n)}(t)
  &=\lambda_{i}^{(n)}(0)+\sqrt{2}W_{t}^{i}+\sum_{j\neq i}\int_{0}^{t}\frac{1}{\lambda_{i}^{(n)}(s)-\lambda_{j}^{(n)}(s)}\ud s,
\end{align}
where $W^{1},\dots, W^{n}$ are independent one-dimensional standard Brownian motions. For further details we refer Anderson et al. \cite{AnGuZe} and Cabanal-Duvillard and Guionnet \cite{CaGu}.

The case where the  $X_{i,j}$'s are fractional Brownian motions of Hurst parameter $H\in(1/2,1)$ was handled in \cite{PaGaPe} using Young integrals and Malliavin calculus techniques. In particular, it was shown that its eigenvalues $(\lambda_{1}^{(n)},\dots, \lambda_{n}^{(n)})$ satisfy a Young integral equation which in turn induces a Skorokhod integral equation for $\langle\mu^{(n)}_t,f\rangle$, when $f\in \Cc_{b}^3(\R)$. Then by taking limits as $n$ tends to infinity in this equation and using some estimations based on Malliavin calculus techniques, one can prove that $\langle\mu^{(n)}_t,f\rangle$ converges to the solution of a deterministic differential equation  which implicitly characterizes the limit process. Similarly to the  Brownian case, the well-posedness of the stochastic Young integral equation for $(\lambda_{1}^{(n)},\dots, \lambda_{n}^{(n)})$ requires the non-collision of the  eigenvalues of $Y^{(n)}$, which was proved by Nualart and Perez-Abreu in \cite{NuPe}. 

As we said before, the previous arguments rely heavily on the fact that the eigenvalues of a fractional Brownian motion with Hurst parameter $H\geq 1/2$ never collide and that a suitable  It\^o or stochastic Young integral equation for $(\lambda_{1}^{(n)},\dots, \lambda_{n}^{(n)})$ can be formulated.  In the case where the $X_{i,j}$'s are general Gaussian processes, these two properties may not hold  and a more refined treatment of the problem is required. Indeed, the non-collision of the eigenvalues for Gaussian processes with highly rough paths, is still an open problem. In addition, if the trajectories of $Y_{i,j}^{(n)}$ are too rough, it is not possible to formulate a stochastic differential equation for its eigenvalues  neither in the It\^o or Young integral sense. In other words  an extended version of the Skorokhod integral is required and consequently the estimations based on Malliavin calculus are harder to handle, since the extended Skorokhod integration doesn't have a clear analogue of  Meyers'  inequality which is required for characterizing the limiting object. 

In the present manuscript, we show that under some mild conditions on the covariance function associated to  $Y^{(n)}$, the process $\langle\mu^{(n)}_t,f\rangle$, for $t\ge 0$, satisfies a Skorokhod stochastic differential  equation (see Lemma 3.1) defined in the extended domain of the divergence (see Section \ref{def:extendeddelta} for a proper definition). In particular, we prove that the Skorokhod stochastic differential  equation  makes sense even in the presence of collision of the eigenvalues. Then we prove a tightness property for the sequence of processes $\{\mu^{(n)} ;n\ge 1\}$ using similar arguments as those presented in \cite{PaGaPe}. It is important to note that due to the lack of a clear analogue of Meyers' inequalities for the extended Skorokhod integral 
deducing the limiting object is not straightforward, in fact we need completely different estimates and techniques to those used in \cite{PaGaPe}.

Our main result  requires the following assumptions on the covariance function  $R$:
\begin{itemize}
\item[\textbf{(H1)}] For every $T>0$, the mapping $s\mapsto R(s,s)$ is continuously differentiable in $(0,\infty)$, continuous at zero and $\frac{d}{ds}R(s,s)$ has finitely many zeros in $(0,T]$. In addition, there exists $\alpha>1$ such that for all $t\in[0,T]$, the mapping $s\mapsto R(s,t)$ is absolutely continuous on $[0,T]$, and 
\begin{align*}
\sup_{0\leq t\leq T}\int_{0}^{T}\Abs{\frac{\partial R}{\partial s}(s,t)}^{\alpha}\ud s<\infty.
\end{align*}
\item[\textbf{(H2)}] There exist constants $\kappa,\gamma>0$, such that for every $s,t>0$, 
$$R(s,s)-2R(s,t)+R(t,t)\leq \kappa \Abs{t-s}^{\gamma}.$$
\end{itemize}

\begin{Theorem}\label{thm:main}
Assume that the covariance function $R$ satisfies conditions \textbf{(H1)} and  \textbf{(H2)}, and $\mu_0^{(n)}$ converges weakly to a probability measure $\mu_0$. Then the family of measure-valued processes $\{\mu^{(n)}: n\ge 1\}$ converges weakly in $\mathcal{C}(\R_{+},\mathtt{Pr}(\R))$ to the unique (deterministic) continuous probability-measure valued function  $(\mu_{t}; t\geq0)$, satisfying  
\begin{align}\label{eq:dynamicsmup}
\Ipm{\mu_{t},f}
  &=\Ipm{\mu_0,f}+\frac{1}{2}\int_{0}^{t}\int_{\R^{2}}\frac{f^{\prime}(x)-f^{\prime}(y)}{x-y}\frac{\ud}{\ud s}R(s,s)\mu_{s}(\ud x)\mu_{s}(\ud y)\ud s,
\end{align}
for each $t\geq0$ and $f\in \mathcal{C}_{b}^{3}$. Moreover, its Cauchy transform $G_{t}(z)$,  satisfies
\begin{align}\label{eq:GandFp}
G_{t}(z)
  &=F_{R(t,t)}(z),\ \ \ \ \ \text{ for } z\in\C_{+} \text{ and } t\geq0,
\end{align}
where $F_{\tau}(z)$ is the unique function differentiable on $\tau$ and analytic on $z$, for $z\in\C_{+}$, satisfying the following Burgers' equation
\begin{align}\label{eq:ICGp}
\frac{\partial}{\partial \tau}F_{\tau}(z)
  &=F_{\tau}(z)\frac{\partial}{\partial z}F_{\tau}(z),\nonumber\\
F_{0}(z)
  &=\int_{\R}\frac{1}{x-z}\mu_0(\ud x).
\end{align}
\end{Theorem}
We note that the term $\frac{f^\prime(x)-f^\prime(y)}{x-y}$ when $x=y$, in the integral of the right-hand side of \eqref{eq:dynamicsmup} is understood as $f^{\prime\prime}(x)$.

Another important observation is related to the fractional Brownian motion. Recall that its  covariance function  satisfies
\begin{equation}\label{aux_2}
R(s,t)=\frac{1}{2}(s^{2H}+t^{2H}-\Abs{t-s}^{2H}),
\end{equation}
 for  $H\in(0,1)$. Such covariance function clearly satisfies  conditions \textbf{(H1)} and \textbf{(H2)} and  consequently, Theorem \ref{thm:main} generalizes the results previously proved for  the Brownian motion in \cite{CeLe} and \cite{RoSh},  and for  the  fractional Brownian motion with Hurst parameter $H\in (1/2,1)$ in  \cite{PaGaPe}.

 Finally, we also point out that in the case that $\mu_0=\delta_0$, the unique solution to \eqref{eq:GandFp} is given by
 	\begin{align*}	
 	F_{\tau}(z)=\frac{1}{2\tau}\left(\sqrt{z^2-4\tau}-z\right),\qquad\text{$t\geq 0$,\,\, $z\in\mathbb{C}_+$.}
 	\end{align*}
 This implies that
 	\begin{align*}
 	G_t(z)=\frac{1}{2R(t,t)}\left(\sqrt{z^2-4R(t,t)}-z\right)\qquad\text{$t\geq 0$,\,\, $z\in\mathbb{C}_+$.}
 	\end{align*}
Hence, for each $t\geq0$, $\mu_t$ is a semicircle distribution with variance $R(t,t)$.

The remainder of this manuscript is organized as follows. In Section \ref{sec:stochastic_calculus} we present some preliminaries on Malliavin calculus and Skorohod integration. In particular, we introduce the extended domain of the divergence. Section \ref{sec:evolution} is devoted to the proof  of the  Skorokhod stochastic differential equation for the process of eigenvalues 
associated to $Y^{(n)}$.  The tightness property is proved in  Section \ref{sec:tightness} and, finally in Section \ref{sec:convergence}, the convergence in law of the sequence $\{\mu^{(n)}; n\geq1\}$ is given.
\section{Preliminaries on Malliavin calculus and  Skorokhod integral}\label{sec:stochastic_calculus} 
Let $d\geq 1$ and $T>0$ be fixed. We denote  by $X=((X_{t}^{1},\dots, X_{t}^{d}); t\in[0,T])$  a  $d$-dimensional continuous Gaussian process defined in a probability space $(\Omega, \Fc,\Pb)$ whose covariance satisfies
\begin{align*}
\E\left[X_{s}^{i}X_{t}^j\right]
  &=\delta_{i,j}R(s,t),\qquad s,t\in[0,T],
\end{align*}
for some non-negative definite covariance function $R$. Denote by $\Ef$ the space of step functions on $[0,T]$. We define in $\Ef$ the scalar product 
\begin{align*}
\Ip{\Indi{[0,s]},\Indi{[0,t]}}_{\Hg}
  &:=  \E\left[X_{s}^{1}X_{t}^{1}\right]\qquad \text{ for }\qquad s,t\in[0,T]. 
\end{align*}
Let $\Hg$ be the Hilbert space obtained by taking the completion of $\Ef$ with respect to this product. For every $1\leq i\leq n$ fixed, the mapping $\Indi{[0,t]} \mapsto X_{t}^{i}$ can be extended to linear isometry between $\Hg$ and the Gaussian subspace of $L^{2}\left(\Omega\right)$ generated by the process $(X_{t}^{i}, t\geq0)$. We will denote this isometry by $X^{i}(h)$, for $h\in\Hg$. 

If $f\in\Hg^{d}$ is of the form $f=(f_{1},\dots, f_{d})$, we set $X(f):=\sum_{i=1}^{d}X^{i}(f_{i})$.
The mapping $f\mapsto X(f)$ is a linear isometry between $\Hg^{d}$ and  the Gaussian subspace of   $L^2\left(\Omega \right)$ generated by $X$. Let $\Sf$ denote the set  of all cylindrical random variables of the form
\begin{align*}
F= g(X(h_{1}),\dots, X(h_{m})),
\end{align*} 
where $g:\R^{m}\rightarrow\R$ is an infinitely differentiable function with compact support, and $h_{j}\in\Ef^{d}$. The Malliavin derivative of $F$ with respect to $X$, is the element of $L^{2}(\Omega;\Hg^{d})$, defined by 
\begin{align}\label{aux_4}
DF
  &=\sum_{i=1}^{m}\frac{\partial g}{\partial x_{i}}(X(h_{1}),\dots, X(h_{m}))h_{i}.
\end{align}
For $p\geq1$, the set $\D^{1,p}$ denotes the closure of $\Sf$ with respect to the norm $\Norm{\cdot}_{\D^{1,p}}$, defined by 
\begin{align*}
\Norm{F}_{\D^{1,p}}
  &:=\Big(\E\left[\Abs{F}^{p}\right]+\E\left[\Norm{DF}_{\Hg^{d}}^{p}\right]\Big)^{\frac{1}{p}}.
\end{align*}
The operator $D$ can be consistently extended to the set $\D^{1,p}$. We denote by $\delta$ the adjoint of the operator $D$, also called the divergence operator. A random element $u\in L^{2}(\Omega;\Hg^{d})$ belongs to the domain of $\delta$, denoted by $\mathrm{Dom} \, \delta $, if and only if satisfies
\begin{align*}
\Abs{\E\left[\Ip{DF,u}_{\Hg^{d}}\right]}
  &\leq C_{u}\E\left[F^{2}\right]^{\frac{1}{2}},\ \text{ for every } F\in\D^{1,2},
\end{align*}
where $C_{u}$ is  a constant only depending on $u$. If $u\in \mathrm{Dom} \,\delta$, then the random variable $\delta(u)$ is defined by the duality relationship
\begin{align*}
\E\left[F\delta(u)\right]=\E\left[\Ip{DF,u}_{\Hg^{d}}\right],
\end{align*}
which holds for every $F\in\D^{1,2}$. We will make use of the notation 
\begin{align}\label{eq:Skorohod_def}
\sum_{i=1}^d\int_{0}^{t} u_{s}^{i}\delta X_{s}^{i}
  &:=\delta(u\Indi{[0,t]}),
\end{align} 
for $u\in L^{2}(\Omega;\Hg^{d})$ of the form $u_t=(u^{1}_{t},\dots, u_{t}^{d})$.

In the case where $X$ is a $d$-dimensional Brownian motion, i.e. its covariance function is given by  $R(s,t)=s\wedge t$ and $\Hg=L^{2}[0,T]$, the random variable \eqref{eq:Skorohod_def} is an extension of the It\^o integral. Motivated by this fact, we may interpret $\sum_{i=1}^d\int_{0}^{t} u_{s}^{i}\delta X_{s}^{i}$ as the stochastic integral of the process $u$. Nevertheless, the space $\Hg$ turns out to be too small for this purpose.  Indeed,  in \cite{ChNu} it was shown that in the case where $X$ is a fractional Brownian motion with Hurst parameter $0<H<\frac{1}{4}$, that is to say its covariance function is of the {form \eqref{aux_2},

the trajectories of $X$ do not belong to the space $\Hg$, and in particular, non-trivial processes of the form $(f(u_{s}); s\in[0,T])$, with $f:\R\rightarrow\R$, might not belong to the domain of $\delta$. 
In order to overcome this difficulty, we extend the domain of $\delta$ by following the approach presented in \cite{NuLe} (see also \cite{ChNu}). The main idea for extending the domain of $\delta$, consists on extending the definition of $\Ip{\varphi,\psi}_{\Hg}$ to the case where $\varphi\in L^{\beta}[0,T]$ for some $\beta>1$, and $\psi$ belongs to the space $\Ef$ of step functions over $[0,T]$. 

In the sequel, we will assume that there exists a constant $\alpha>1$ such that the following condition holds. Let $\beta$ be the conjugate of $\alpha$, defined by $\beta:=\alpha/(\alpha-1)$. For any pair of functions $\varphi\in L^{\beta}[0,T]$ and $\psi\in\Ef$ of the form $\psi=\sum_{j=1}^{m}c_{j}\Indi{[0,t_{j}]}$, we define
\begin{align}\label{def:extendedIP}
\Ip{\psi,\varphi}_{\Hg}
  &:=\sum_{j=1}^{m}c_j\int_{0}^{T}\varphi(s)\frac{\partial R}{\partial s}(s,t_{j})\ud s.
\end{align}
This expression is well defined since
\begin{align*}
\Abs{\Ip{\Indi{[0,t]},\varphi	}_{\Hg}}
  &=\Abs{\int_{0}^{T}\varphi_{s}\frac{\partial R}{\partial s}(s,t)\ud s}
	\leq\Norm{\varphi}_{L^{\beta}[0,T]}\sup_{0\leq t\leq T}\left(\int_{0}^{T}\Abs{\frac{\partial R}{\partial s}(s,t)}^{\alpha}\ud s\right)^{\frac{1}{\alpha}}<\infty,
\end{align*}
and coincides with the inner product in $\Hg$ in the case where $\varphi\in\Ef$. Indeed, for $\varphi\in\Ef$ of the form $\varphi=\sum_{i=1}^{n}a_{i}\Indi{[0,t_{i}]},$ we have
\begin{align*}
\Ip{\Indi{[0,t]},\varphi}_{\Hg}
  &=\sum_{i=1}^{n}a_{i}R(t_{i},t)
	=\sum_{i=1}^{n}a_{i}\int_{0}^{t_{i}}\frac{\partial R}{\partial s}(s,t)\ud s=\int_{0}^{T}\varphi(s)\frac{\partial R}{\partial s}(s,t)\ud s.
\end{align*}
We define the extended domain of the divergence as follows.
\begin{Definition}\label{def:extendeddelta}
Let $\Ip{\cdot,\cdot}_{\Hg}$ be the bilinear function defined by \eqref{def:extendedIP}. We say that a stochastic process $u\in L^{1}(\Omega;L^{\beta}[0,T])$ belongs to the extended domain of the divergence $Dom^{*}\delta$ if there exists $p>1$, such that 
\begin{align*}
\Abs{\E\left[\Ip{DF,u}_{\Hg^{d}}\right]}
  &\leq C_{u}\Norm{F}_{L^{p}(\Omega)},
\end{align*}
for any smooth random variable $F\in\Sf$, where $C_{u}$ is some constant depending on $u$. In this case, $\delta(u)$ is defined by the duality relationship
\begin{align}\label{eq:Skorohodduality}
\E\left[F\delta(u)\right]=\E\left[\Ip{DF,u}_{\Hg^{d}}\right].
\end{align}
\end{Definition}

It is important to note that for a general covariance function $R(s,t)$ and $\beta>1$, the domains $\text{Dom}^{*}\delta$ and $\text{Dom} \delta$ are not necessarily comparable (see Section 3 in \cite{NuLe} for further details about this fact).

The next result is a multidimensional version of  It\^o's formula for the Skorokhod integral and for functions that are smooth only on a dense open subset of the Euclidean space and satisfy some extra regularity conditions. In the sequel, for every $i\in\{1,\dots, d\}$, the map $\gamma_{i}:\R^{d}\rightarrow\R^{d-1}$ denotes the projection over the hyperplane $P_{i}:=\{(x_{1},\dots, x_{d})\ |\ x_{i}=0\}$. We  use as well the following notation: for every real function $h:\mathcal{D}\subset\R^{d}\rightarrow\R$,  we define
\begin{align*}
\mathbf{1}_{\mathcal{D}}(x)h(x)
  &:=\left\{\begin{array}{cc}h(x)&\ \ \text{ if }\ x\in \mathcal{D},\\0&\ \ \ \ \ \ \text{ if } \ x\in \R^d\backslash \mathcal{D},\end{array}\right.
\end{align*}
for every $x\in \R^{d}$.

\begin{Theorem}\label{thm:chainrule}
Assume that $R$ satisfies \textbf{(H1)}.  Consider a function $F:\R^{d}\rightarrow\R$, with $d>1$, satisfying the following conditions:
\begin{enumerate}
\item There exists a measurable set $M\subset \R^{d}$, with Lebesgue measure zero, such that $F$ is twice continuously differentiable in $\mathcal{D}:=\R^d\backslash M$ and $\gamma_{i}(M)$ has measure zero with respect to the Lebesgue measure in $P_{i}$.
\item There exist constants $C>0$ and $N>0$, such that for all $x\in \mathcal{D}$ and $i\in\{1,\dots,d\},$ 
\begin{align}\label{cond1:thmchain}
\Abs{F(x)}+\Abs{\frac{\partial F}{\partial x_{i}}(x)}\leq C(1+\Abs{x}^{N}).
\end{align}
\item There exists $0<\delta<1$, such that for every $p\geq1$, $i\in\{1,\dots, d\}$ and $s>0$, the random variable
$|\frac{\partial^2 F}{\partial x_{i}^2}(X_{s})|$ has finite expectation, and 
\begin{align}\label{cond2:thmchain}
\E\left[\Abs{\sum_{i=1}^{d}\frac{\partial^2 F}{\partial x_{i}^2}(X_{s})}^p\right]
  &\leq C\Big(1+R(s,s)^{-p(1-\delta)}\Big),
\end{align}
for some constant $C>0$. 
\end{enumerate}
Then,  the process $u_{s}=(u_{s}^{1},\dots, u_{s}^{d})$ defined by  $u_{s}^{i}:=\mathbf{1}_{\mathcal{D}}(X_s)\frac{\partial F}{\partial x_{i}}(X_s)\mathbf{1}_{[0,t]}(s),$
belongs to $\text{Dom}^*\delta$, and 
\begin{align}\label{eq:Itotypeformula}
F(X_t)
	&=F(X_0)+\sum_{i=1}^{d}\int_{0}^{t}\mathbf{1}_{\mathcal{D}}(X_s)\frac{\partial F}{\partial x_{i}}(X_s)\delta X_s^{i}+\frac{1}{2}\sum_{i=1}^{d}\int_{0}^{t}\mathbf{1}_{\mathcal{D}}(X_s)\frac{\partial^2 F}{\partial x_{i}^2}(X_s)\frac{\ud R(s,s)}{\ud s}\ud s,
\end{align}
for every $t\in[0,T].$
\end{Theorem}
\noindent Before proving this result, we provide some interesting remarks.\\

\noindent (i) For every $T>0$, with probability one, the random set $I_{T}=\{s\in[0,T]\ |\ X_{s}\in M\}$ has Lebesgue measure $|I_{T}|$ equal to zero, since 
\begin{align*}
\E\left[|I_{T}|\right]
  &=\E\left[\int_{0}^{T}\mathbf{1}_{\{X_{s}\in M\}}\ud s\right]=\int_{0}^T\Pb[X_{s}\in M]\ud s=0,
\end{align*}
where the last equality follows from the fact that $M$ has Lebesgue measure zero and $X_{s}$ has a Gaussian distribution. As a consequence, with probability one the trajectories of $\mathbf{1}_{\mathcal{D}}(X_{s})$ are Lebesgue almost everywhere equal to one, which allows us to rewrite equation \eqref{eq:Itotypeformula} as follows
\begin{align*}
F(X_{t})
	&=F(X_0)+\sum_{i=1}^{d}\int_{0}^{t}\frac{\partial F}{\partial x_{i}}(X_{s})\delta X_{s}^{i}+\frac{1}{2}\sum_{i=1}^{d}\int_{0}^{t}\frac{\partial^2 F}{\partial x_{i}^2}(X_{s})\frac{\ud R(s,s)}{\ud s}\ud s,
\end{align*}
with the understanding that, although the integrands might be undefined for some values of $s$, they are well defined Lebesgue almost everywhere. Nevertheless, we will use the notation \eqref{eq:Itotypeformula}, in order to avoid confusion.

(ii) A version of the previous result was first presented in \cite[Theorem~3.1]{NuPe} for $d\geq 1$, where the condition (1) was replaced by the weaker condition that $f$ is differentiable in an open dense set $\mathcal{D}$ of $\R^{d}$. Unfortunately, this result is false, as we can verify by taking $d=1$, $f(x)=\Abs{x}$, and covariance $R(s,t)=s\wedge t$, which corresponds to the standard Brownian motion. Under these conditions, the third term appearing in the right hand side of \eqref{eq:Itotypeformula} must be replaced by the local time of the Brownian motion.  In order for the result to hold, we require the more restrictive condition (1) instead of the differentiability of $f$ over an open dense set.

(iii) Condition (3) is slightly more general than the one presented in \cite[Theorem~3.1]{NuPe}. This generalization is crucial for providing a Skorohod integral equation for $\langle \mu_t^{(n)},f\rangle$, since in this case the function $\sum_{k\leq h}\frac{\partial^2 F}{\partial^{2}x_i}$ is smooth and bounded, unlike the individual components $\frac{\partial^2 F}{\partial^{2}x_i}$, which are considerably more erratic (see Section \ref{sec:evolution} for details).\\

The proof that we present below is based on similar arguments as those used in \cite[Theorem~3.1]{NuPe}, but some modifications and additional techniques are required.

\begin{proof}[Proof of Theorem \ref{thm:chainrule}]
Let $Y\in\Sf$ be of the form 
\begin{align*}
Y
  &=\tilde{g}(V(h_{1}),\dots, V(h_{q})),
\end{align*}
for $h_{i}=(h_{i}^{1},\dots, h_{i}^{d})$, with $h_{i}^{l}\in\Ef$ and $\tilde{g}:\R^{qd}\rightarrow\R$ infinitely differentiable with compact support. Since each $h_{i}^{l}$ is a step function of the form
\begin{align*}
h_{i}^{l}(x)
  &=\sum_{j=1}^{r}a_{i,j}^{l}\mathbf{1}_{[0,s_{i,j}^l)}(x),
\end{align*}
for some $r\in\N$, $a_{i,j}^l\in\R$ and $s_{i,j}^{l}\in\R_{+}$ for $1\leq j\leq r$, we deduce that there exist $m\in\N$ and $t_{1},\dots, t_{m}\in\R_{+}$, such that 
\begin{align*}
Y=g(X_{t_{1}},\dots ,X_{t_{m}}),
\end{align*}
for some $g:\R^{md}\rightarrow\R$ infinitely differentiable with compact support. 
Using the chain rule for $D$, we obtain
\begin{align*}
DY
  &=\sum_{i=1}^{m}\sum_{j=1}^{d}\frac{\partial g}{\partial y_{i,j}}(X_{t_{1}},\dots, X_{t_{m}})DX_{t_{i}}^{j},
\end{align*}
where $\frac{\partial g}{\partial y_{i,j}}(v_{1},\dots, v_{m})$ denotes the partial derivative of $g$ with respect to the $j$-th component of $v_{i}$, evaluated at $(v_{1},\dots, v_{m})$. By condition \eqref{cond1:thmchain} and the way we choose the process $u_s$, the inner product $\Ip{\mathbf{1}_{[0,t_{i}]},u^{j}}_{\Hg}$ is well defined and satisfies
\begin{align*}
\Ip{\mathbf{1}_{[0,t_{i}]},u^{j}}_{\Hg}
  &=\big\langle\mathbf{1}_{[0,t_{i}]},\mathbf{1}_{\mathcal{D}}(X_{\cdot})\frac{\partial F}{\partial x_{j}}(X_{\cdot})\mathbf{1}_{[0,t]}(\cdot)\big\rangle_{\Hg}
	  =\int_{0}^{t}\mathbf{1}_{\mathcal{D}}(X_s)\frac{\partial F}{\partial x_{j}}(X_{s})\frac{\partial R}{\partial s}(s,t_{i})\ud s.
\end{align*}
Hence, using the fact that $\Ip{DX_{t_{i}}^{j},u}_{\Hg^d}=\Ip{\mathbf{1}_{[0,t_{i}]},u^{j}}_{\Hg}$, we get
\begin{align*}
\Ip{DY,u}_{\Hg^{d}}
	&=\sum_{i=1}^{m}\sum_{j=1}^{d}\frac{\partial g}{\partial y_{i,j}}(X_{t_{1}},\dots, X_{t_{m}})\int_{0}^{t}\mathbf{1}_{\mathcal{D}}(X_{s})\frac{\partial F}{\partial x_{j}}(X_{s})\frac{\partial R}{\partial s}(s,t_{i})\ud s.
\end{align*}
Using the previous expression as well as \eqref{cond1:thmchain}, we deduce that $\Ip{DY,u}_{\Hg^{d}}$ is integrable. Indeed, since $g$ is compactly supported, we can use \eqref{cond1:thmchain} to obtain a constant $C>0$ such that
\begin{align}\label{ineq:DYuintegrable}
\E\left[\Abs{\Ip{DY,u}_{\Hg^{d}}}\right]
	&\leq C\sum_{i=1}^{m}\sum_{j=1}^{d}\int_{0}^{T}\E\left[(1+\Abs{X_{s}}^N)\right]\Abs{\frac{\partial R}{\partial s}(s,t_{i})}\ud s.
\end{align}
Moreover, since $X_{s}$ is Gaussian and the mapping $s\mapsto R(s,s)$ is continuous in $[0,T]$,
$$\E\left[\Abs{X_{s}}^{N}\right]\leq N!!R(s,s)^{\frac{N}{2}}\leq  N!!\sup_{0\leq s\leq T}R(s,s)^{\frac{N}{2}}<\infty,$$ 
where $N!!$ denotes the double factorial of $N$. The integrability of $\Abs{\Ip{DY,u}_{\Hg^{d}}}$ then  follows from \eqref{ineq:DYuintegrable} and condition \textbf{(H1)}. As a consequence, we can write
\begin{align}\label{eq:EDyu}
\E\left[\Ip{DY,u}_{\Hg^{d}}\right]
  &=\sum_{i=1}^{m}\sum_{j=1}^{d}\int_{0}^{t}\E\left[\frac{\partial g}{\partial y_{i,j}}(X_{t_{1}},\dots, X_{t_{m}})\mathbf{1}_{\mathcal{D}}(X_{s})\frac{\partial F}{\partial x_{j}}(X_{s})\right]\frac{\partial R}{\partial s}(s,t_{i})\ud s\nonumber\\
	&=\sum_{i=1}^{m}\sum_{j=1}^{d}\int_{0}^{t}\int_{\R^{dm}}\int_{\R^{d}}\frac{\partial g}{\partial y_{i,j}}(y)\mathbf{1}_{\mathcal{D}}(x)\frac{\partial F}{\partial x_{j}}(x)\frac{\partial R}{\partial s}(s,t_{i})f_s(x,y)\ud x\,\ud y\,\ud s,
\end{align}
where $f_s:\R^{d(m+1)}\rightarrow\R_{+}$ denotes the join density of the Gaussian vector $(X_{s},X_{t_{1}},\dots X_{t_{m}})$. Let $x$ be of the form $x=(x_{1},\dots, x_{d})$. Since $\gamma_{j}(M)$ has measure zero in $P_{j}$, we deduce that for every $y\in\R^{md}, s>0$ and $j\geq1$,
\begin{align}\label{eq:previntbyparts}
\int_{\R^{d}}\mathbf{1}_{\mathcal{D}}(x)\frac{\partial F}{\partial x_{j}}(x)f_s(x,y)\ud x
  &=\int_{P_{j}}\int_{\R}\mathbf{1}_{\mathcal{D}}(x)\frac{\partial F}{\partial x_{j}}(x)f_s(x,y)\ud x_{j}\prod_{i\neq j}\ud x_{i}\nonumber\\
	&=\int_{\gamma_{j}(M)^c}\int_{\R}\mathbf{1}_{\mathcal{D}}(x)\frac{\partial F}{\partial x_{j}}(x)f_s(x,y)\ud x_{j}\prod_{i\neq j}\ud x_{i}.
\end{align}
By condition (1), for every $(x_{1},\dots,x_{j-1},x_{j+1},\dots, x_{d})\in \gamma_{i}(M)^{c}$, the mapping 
$$t\mapsto F(x_{1},\dots,x_{j-1},t,x_{j+1},\dots, x_{d})$$ 
is differentiable in $\R$, and hence, using the polynomial growth of $F$ and $\frac{\partial F}{\partial x_{j}}$, we can remove the term $\mathbf{1}_{\mathcal{D}}(x)$ in the right hand side of \eqref{eq:previntbyparts}, and integrate by parts the variable $x_{i}$, in order to deduce
\begin{align}\label{eq:intbyparts}
\int_{\R^{d}}\frac{\partial F}{\partial x_{j}}(x)f_s(x,y)\ud x
  &=\int_{\gamma_{j}(M)^c}\int_{\R}\frac{\partial F}{\partial x_{j}}(x)f_s(x,y)\ud x_{j}\prod_{i\neq j}\ud x_{i}\nonumber\\
  &=-\int_{\gamma_{j}(M)^c}\int_{\R}F(x)\frac{\partial f_s}{\partial x_{j}}(x,y)\ud x_{j}\prod_{i\neq j}\ud x_{i}\nonumber\\
	&=-\int_{\R^d}\mathbf{1}_{\mathcal{D}}(x)F(x)\frac{\partial f_s}{\partial x_{j}}(x,y)\ud x.
\end{align}

From \eqref{eq:EDyu} and \eqref{eq:intbyparts}, we conclude that
\begin{align}\label{eq:IBP1}
\E\left[\Ip{DY,u}_{\Hg^{d}}\right]
  &=-\sum_{i=1}^{m}\sum_{j=1}^{d}\int_{0}^{t}\int_{\R^{dm}}\int_{\R^{d}}\frac{\partial g}{\partial y_{i,j}}(y)\mathbf{1}_{\mathcal{D}}(x)F(x)\frac{\partial R}{\partial s}(s,t_{i})\frac{\partial f_s}{\partial x_{j}}(x,y)\ud x\,\ud y\,\ud s\nonumber\\
	&=\sum_{i=1}^{m}\sum_{j=1}^{d}\int_{0}^{t}\int_{\R^{dm}}\int_{\R^{d}}g(y)\mathbf{1}_{\mathcal{D}}(x)F(x)\frac{\partial R}{\partial s}(s,t_{i})\frac{\partial^2 f_s}{\partial y_{i,j}\partial x_{j}}(x,y)\ud x\,\ud y\,\ud s.
\end{align}
Similarly, using relation \eqref{cond2:thmchain}, as well as the fact that $g$ is compactly supported, we have that for every $1\leq j\leq d$, the random variable 
\begin{equation}\label{therv}
Y\int_{0}^{t}\mathbf{1}_{\mathcal{D}}(X_{s})\sum_{j=1}^{d}\frac{\partial^{2}F}{\partial x_{j}^{2}}(X_{s})\frac{\ud}{\ud s}R(s,s)\ud s
\end{equation}
 is integrable. Indeed, by \eqref{cond2:thmchain}, there exists a constant $C>0$ such that
\begin{align}\label{ineq:Yd2Fdrintegral}
\E\left[\Abs{Y\int_{0}^{t}\mathbf{1}_{\mathcal{D}}(X_s)\sum_{j=1}^{d}\frac{\partial^{2}F}{\partial x_{j}^{2}}(X_{s})\frac{\ud}{\ud s}R(s,s)\ud s}\right]
  &\leq C\int_{0}^{T}\E\left[\Abs{\sum_{j=1}^{d}\frac{\partial^{2}F}{\partial x_{j}^{2}}(X_{s})}\right]\Abs{\frac{\ud}{\ud s}R(s,s)}\ud s\nonumber\\
  &\leq C\int_{0}^{T}(1+R(s,s)^{-1+\delta})\Abs{\frac{\ud}{\ud s}R(s,s)}\ud s.
\end{align}
By condition \textbf{(H1)}, there exist $L\in\N$ and $0=T_{1}<\cdots <T_{L}=T$, such that $R(s,s)$ is monotone in $[T_{i},T_{i+1}]$ for all $1\leq i\leq L-1$. 
Hence,
\begin{align}\label{altineq:dRss}
\int_{0}^{T}(1+R(s,s)^{-1+\delta})\Abs{\frac{\ud R(s,s)}{\ud s}}\ud s
  &= \sum_{i=1}^{L-1}\bigg|\int_{T_{i}}^{T_{i+1}}(1+R(s,s)^{-1+\delta})\frac{\ud R(s,s)}{\ud s}\ud s\bigg|\nonumber\\
	&= \sum_{i=1}^{L-1}\big|R(T_{i+1},T_{i+1})-R(T_{i},T_{i})\nonumber\\
	&+\frac{1}{\delta}(R(T_{i+1},T_{i+1})^\delta-R(T_{i},T_{i})^\delta)\big|.
\end{align}
Therefore, by \eqref{ineq:Yd2Fdrintegral}, the random variable in \eqref{therv} is integrable, as required. Proceeding as in the proof of \eqref{eq:intbyparts} and using the fact that $\E\left[\Abs{\frac{\partial^{2}F}{\partial x_{j}^{2}}(X_{s})}\right]<\infty$ for all $s>0$, we can show that for all $y\in\R^{md}$ and $s>0$,
\begin{align*}
\int_{\R^{d}}\mathbf{1}_{\mathcal{D}}(x)\frac{\partial^2 F}{\partial x_{j}^2}(x)f_{s}(x,y)\ud x
	&=\int_{\R^{d}}\mathbf{1}_{\mathcal{D}}(x)F(x)\frac{\partial^2 f_{s}}{\partial x_{j}^2}(x,y)\ud x,
\end{align*}
and consequently,
\begin{multline}\label{eq:IBP3}
\E\left[Y\int_{0}^{t}\mathbf{1}_{\mathcal{D}}(X_s)\sum_{j=1}^d\frac{\partial^{2}F}{\partial x_{j}^{2}}(X_{s})\frac{\ud}{\ud s}R(s,s)\ud s\right]\\
\begin{aligned}
  &=\int_{0}^{t}\int_{\R^{dm}}\int_{\R^{d}}g(y)\mathbf{1}_{\mathcal{D}}(x)\sum_{j=1}^d\frac{\partial^2 F}{\partial x_{j}^2}(x)f_{s}(x,y)\frac{\ud}{\ud s}R(s,s)\ud x\,\ud y\,\ud s\\
	&=\int_{0}^{t}\int_{\R^{dm}}\int_{\R^{d}}g(y)\mathbf{1}_{\mathcal{D}}(x)\sum_{j=1}^dF(x)\frac{\partial^2 f_{s}}{\partial x_{j}^2}(x,y)\frac{\ud}{\ud s}R(s,s)\ud x\,\ud y\,\ud s.
\end{aligned}
\end{multline}
In addition, by \eqref{cond1:thmchain}, the random variable $YF(X_{t})$ is integrable for every $t\geq0$, and
\begin{align}\label{eq:IBP2}
\E\left[YF(X_{t})-YF(X_0)\right]
  &=\int_{\R^{dm}}\int_{\R^{d}}g(y)F(x)f_{t}(x,y)\ud x\,\ud y-\int_{\R^{dm}}\int_{\R^{d}}g(y)F(x)f_{0}(x,y)\ud x\,\ud y\nonumber\\
	&=\int_{0}^{t}\int_{\R^{dm}}\int_{\R^{d}}g(y)F(x)\frac{\partial f_{s}}{\partial s}(x,y)\ud x\,\ud y\,\ud s\nonumber\\
	&=\int_{0}^{t}\int_{\R^{dm}}\int_{\R^{d}}g(y)\mathbf{1}_{\mathcal{D}}(x)F(x)\frac{\partial f_{s}}{\partial s}(x,y)\ud x\,\ud y\,\ud s,
\end{align}
where the last identity follows from the fact that $\R\backslash \mathcal{D}$ has Lebesgue measure zero. Finally, we have
\begin{align}\label{eq:gausskerheateq}
\frac{\partial f_{s}}{\partial s}(x,y)
  &=\sum_{i=1}^{m}\sum_{j=1}^{d}\frac{\partial R}{\partial s}(s,t_{i})\frac{\partial^2 f_s}{\partial y_{i,j}\partial x_{j}}(x,y)
	+\frac{1}{2}\sum_{j=1}^{d}\frac{\ud}{\ud s}R(s,s)\frac{\partial^2 f_{s}}{\partial x_{j}^2}(x,y).
\end{align}
From \eqref{eq:IBP1},\eqref{eq:IBP3},\eqref{eq:IBP2} and \eqref{eq:gausskerheateq}, we get that 
\begin{align}\label{eq:chainaux}
\E\left[YF(X_{t})-YF(X_0)\right]
  &=\E\left[\Ip{DY,u}_{\Hg^{d}}\right]+\frac{1}{2}\E\left[Y\int_{0}^{t}\mathbf{1}_{\mathcal{D}}(X_s)\sum_{j=1}^{d}\frac{\partial^{2}F}{\partial x_{j}^{2}}(X_{s})\frac{\ud}{\ud s}R(s,s)\ud s\right].
\end{align}
Next we use \eqref{eq:chainaux} to prove that $u$ belongs to the extended domain of the divergence $\text{Dom}^{*}\delta$.  Using H\"older inequality and Minkowski inequality, we have that
\begin{multline*}
\Abs{\E\left[Y\int_{0}^{t}\mathbf{1}_{\mathcal{D}}(X_s)\sum_{j=1}^{d}\frac{\partial^{2}F}{\partial x_{j}^{2}}(X_{s})\frac{\ud}{\ud s}R(s,s)\ud s\right]}\\
\begin{aligned}
	&\leq \Norm{Y}_{L^{\frac{p}{p-1}}(\Omega)}\Norm{\int_{0}^{t}\mathbf{1}_{\mathcal{D}}(X_s)\sum_{j=1}^{d}\frac{\partial^{2}F}{\partial x_{j}^{2}}(X_{s})\frac{\ud}{\ud s}R(s,s)\ud s}_{L^{p}(\Omega)}\\
	&\leq \Norm{Y}_{L^{\frac{p}{p-1}}(\Omega)}\int_{0}^{t}\Norm{\sum_{j=1}^{d}\frac{\partial^{2}F}{\partial x_{j}^{2}}(X_{s})}_{L^{p}(\Omega)}\Abs{\frac{\ud}{\ud s}R(s,s)}\ud s.
\end{aligned}
\end{multline*}
In addition, using \eqref{cond2:thmchain}, we get 
\begin{align*}
\int_{0}^{T}\Norm{\frac{\partial^{2}F}{\partial x_{j}^{2}}(X_{s})}_{L^{p}(\Omega)}\Abs{\frac{\ud}{\ud s}R(s,s)}\ud s
	&\leq C\int_{0}^{T}(1+R(s,s)^{-p(1-\delta)})\Abs{\frac{\ud}{\ud s}R(s,s)}\ud s.
\end{align*}
%
Thus, if $1<p<\frac{1}{1-\delta}$, by replacing $\delta$ for $1-p(1-\delta)$ in \eqref{altineq:dRss}, we obtain
\begin{align}\label{ineq:d2Fdr}
\int_{0}^{t}\Norm{\frac{\partial^{2}F}{\partial x_{j}^{2}}(X_{s})}_{L^{p}(\Omega)}\Abs{\frac{\ud}{\ud s}R(s,s)}\ud s
&\leq C\int_{0}^{T}(1+R(s,s)^{-p(1-\delta)})\Abs{\frac{\ud}{\ud s}R(s,s)}\ud s
<\infty.
\end{align}	

Therefore, using H\"older's inequality, as well as \eqref{eq:chainaux} and \eqref{ineq:d2Fdr}, we deduce that there exists a constant $C>0$ such that
\begin{align*}
\Abs{\E\left[\Ip{DY,u}_{\Hg^{d}}\right]}
  &\leq\Abs{\E\left[YF(X_0)\right]}+\frac{1}{2}\Abs{\E\left[Y\int_{0}^{t}\mathbf{1}_{\mathcal{D}}(X_{s})\frac{\partial^{2}F}{\partial x_{j}^{2}}(X_{s})\frac{\ud}{\ud s}R(s,s)\ud s\right]}+\Abs{\E\left[YF(X_{t})\right]}\\
	&\leq C\left(\Norm{Y}_{L^{1}(\Omega)}+\Norm{Y}_{L^{\frac{p}{p-1}}(\Omega)}+\E\left[\Abs{Y}(1+\Abs{X_{t}}^{N})\right]\right)\\
	&\leq C\left(\Norm{Y}_{L^{1}(\Omega)}+\Norm{Y}_{L^{\frac{p}{p-1}}(\Omega)}+\Norm{Y}_{L^{\frac{p}{p-1}}(\Omega)}\Norm{1+\Abs{X_{t}}^{N}}_{L^{p}(\Omega)}\right),
\end{align*}
and hence, there exists a constant $C>0$ such that 
\begin{align}\label{ineq:domdivu}
\Abs{\E\left[\Ip{DY,u}_{\Hg^{d}}\right]}
  &\leq C\Norm{Y}_{L^{\frac{p}{p-1}}(\Omega)}.
\color[rgb]{0,0,0}\end{align}
From \eqref{eq:chainaux} and \eqref{ineq:domdivu}, it follows that $u$ belongs to the extended domain of the divergence $\text{Dom}^{*}\delta$, and 
\begin{align*}
F(X_{t})-F(X_0)
  &=\sum_{i=1}^{d}\int_{0}^{t}\mathbf{1}_{\mathcal{D}}(X_{s})\frac{\partial F}{\partial x_{i}}(X_{s})\delta X_{s}^{i} + \frac{1}{2}\sum_{i=1}^{d}\int_{0}^{t}\mathbf{1}_{\mathcal{D}}(X_{s})\frac{\partial^2 F}{\partial x_{i}^2}(X_{s})\frac{\ud}{\ud s}R(s,s)\ud s,
\end{align*}
as required.
\end{proof}
\section{Stochastic Evolution of the eigenvalues of a matrix-valued Gaussian process}\label{sec:evolution}
We first recall some notation. Consider a family of independent and identically distributed centered Gaussian processes $\{X_{i,j}; i,j\in\N\}$ defined in a probability space $(\Omega,\Fc,\Pb)$. We will assume that the covariance function $R(s,t):=\E\left[X_{1,1}(s)X_{1,1}(t)\right]$ satisfies the hypotheses \textbf{(H1)} and \textbf{(H2)}. Consider as well a sequence of deterministic symmetric matrices $A^{(n)}=[A_{i,j}^{(n)}]_{1\leq i,j\leq n}$, with ordered eigenvalues $\lambda_{1}^{(n)}(0)\geq\cdots\geq \lambda_{n}^{(n)}(0)$ and spectral empirical distribution 
\[
\mu_0^{(n)}:=\frac{1}{n}\sum_{i=1}^{n}\delta_{\lambda_i^{(n)}(0)} \qquad\textrm{satisfying}\qquad
\mu_0^{(n)}
  \stackrel{\mathcal{L}}{\rightarrow}\mu_0\qquad\text{as $n\to\infty$},
\]
for some probability law $\mu_0$ and where $\stackrel{\mathcal{L}}{\rightarrow}$ means convergence in law. Let $Y^{(n)}=[Y_{i,j}^{(n)}]_{1\leq i,j\leq n}$ be the renormalized symmetric Gaussian matrix of dimension $n\times n$, given by
\begin{equation}\label{eq:scaledB}
Y_{i,j}^{(n)}(t)
  :=\left\{\begin{array}{cc}
	\displaystyle\frac{1}{\sqrt{n}}X_{i,j}(t) +A_{i,j}^{(n)}&\ \ \text{ if }\ i<j,\\
	\displaystyle\frac{\sqrt{2}}{\sqrt{n}}X_{i,i}(t) + A_{i,i}^{(n)}&\ \ \text{ if }\ i=j,
	\end{array}\right.
	\end{equation}
Denote by $\lambda_{1}^{(n)}(t)\geq\cdots\geq\lambda_{n}^{(n)}(t)$ the ordered eigenvalues of $Y^{(n)}(t)$, and by $\mu^{(n)}:=(\mu_{t}^{(n)}, t\geq0)$ the corresponding empirical measure process
\begin{align*}
\mu_{t}^{(n)}
  &=\frac{1}{n}\sum_{j=1}^{n}\delta_{\lambda_{j}^{(n)}(t)}.
\end{align*}
For a given probability measure $\nu$, and a $\nu$-integrable function $f$, we use the notation 
$\Ip{\nu,f}:=\int f(x)\nu(\ud x).$
In particular,
\begin{align}\label{eq:muaction}
\Ipn{\mu_{t}^{(n)},f}
  &=\frac{1}{n}\sum_{i=1}^{n}f(\lambda_{i}^{(n)}(t)).
\end{align}
From \cite[Lemma~5.1]{NuPe}, it follows that for every $i=1,\dots, n$, there exists a function $\Phi_{i}^n:\R^{\frac{n(n+1)}{2}}\rightarrow\R$, which is infinitely differentiable in an open subset $G\subset\R^{\frac{n(n+1)}{2}}$, with $|G^{c}|=0$, such that $\lambda_{i}^{(n)}(t)=\Phi_{i}^n(Y^{(n)}(t))$. Moreover, every element $X\in G$, viewed as an $n\times n$ matrix, has a factorization of the form $Z^{(n)}=U^{(n)} {\tt D}(U^{(n)})^{*}$, where ${\tt D}$ is a diagonal matrix with entries ${\tt D}_{i,i}=\lambda_{i}^{(n)}$ such that $\lambda_{1}^{(n)}>\dots>\lambda_{n}^{(n)}$, $U^{(n)}$ is an orthogonal matrix with $U^{(n)}_{i,i}>0$ for all $i$, $U^{(n)}_{i,j}\neq0$ and all the minors of $U^{(n)}$ have non zero determinants. In addition, for any $k\leq h$, we have 
\begin{align}
\frac{\partial \Phi_{i}^n}{\partial y_{k,h}}(Z^{(n)})
  &=2 U^{(n)}_{i,k}U^{(n)}_{i,h}\mathbf{1}_{\{k\neq h\}} + \sqrt{2}(U^{(n)}_{i,k})^2\mathbf{1}_{\{k=h\}},\label{eq:D1PhiU}\\
\frac{\partial^{2}\Phi_{i}^n}{\partial y_{k,h}^2}(Z^{(n)}),
  &=2\sum_{j\neq i}\frac{\Abs{U^{(n)}_{i,k}U^{(n)}_{j,h}+U^{(n)}_{i,h}U^{(n)}_{j,k}}^2}{\lambda_{i}^{(n)}-\lambda_{j}^{(n)}}\mathbf{1}_{\{k\neq h\}}
	+4\sum_{j\neq i}\frac{\Abs{U^{(n)}_{i,k}U^{(n)}_{j,k}}^2}{\lambda_{i}^{(n)}-\lambda_{j}^{(n)}}\mathbf{1}_{\{k=h\}}\label{eq:D2PhiU}.
\end{align}
Using the orthogonality of the columns of $U^{(n)}$, we deduce  from \eqref{eq:D1PhiU} and \eqref{eq:D2PhiU} that
\begin{align}
\sum_{k\leq h}\frac{\partial^{2}\Phi_{i}^n}{\partial y_{k,h}^2}(Z^{(n)})
  =\sum_{j\neq i}\frac{2}{\lambda_{i}^{(n)}-\lambda_{j}^{(n)}}\ \ \ \ \ \ \ \ \ \text{ and }\ \ \ \ \ \ \ \ \ \sum_{k\leq h}\left(\frac{\partial \Phi_{i}^n}{\partial y_{k,h}}(Z^{(n)})\right)^2
  =2\label{eq:D2Phisum}.
\end{align}
Using Lemma \ref{thm:chainrule}, we can prove the following result, which describes the time evolution of the eigenvalues $(\lambda_{1}^{(n)}(t),\dots,\lambda_{n}^{(n)}(t))$ in terms of the Skorohod integral.

\begin{Lemma}\label{lem:evolution}
For every $f\in \Cc_{b}^2(\R)$ and $t\geq0$, we have
\begin{align}\label{eq:evolution}
\big\langle\mu_{t}^{(n)},f\big\rangle
  &=\big\langle\mu_{0}^{(n)},f\big\rangle +\frac{1}{2n^{2}}\sum_{i=1}^{n}\int_{0}^{t}f^{\prime\prime}(\Phi_{i}^{n}(Y^{(n)}(s)))\frac{\ud}{\ud s}R(s,s)\ud s\nonumber\\
  &+\frac{1}{n^{\frac{3}{2}}}\sum_{i=1}^{n}\sum_{k\leq h}\int_{0}^{t}f^{\prime}(\Phi_{i}^{n}(Y^{(n)}(s)))\mathbf{1}_{G}(Y^{(n)}(s))\frac{\partial \Phi_{i}^{n}}{\partial y_{k,l}}(Y^{(n)}(s))\delta X_{k,h}(s)\\
	&+\frac{1}{2}\int_{0}^{t}\int_{\R^2}\frac{f^{\prime}(x)-f^{\prime}(y)}{x-y}\frac{\ud}{\ud s}R(s,s)\mu_{s}^{(n)}(\ud x)\mu_{s}^{(n)}(\ud y)\ud s.\nonumber
\end{align}
\end{Lemma}
\begin{proof}
For simplicity, we introduce $n^{-1/2}X^{(n)}:=Y^{(n)}-A^{(n)}$ and we write
\begin{align}\label{eq:muf1}
\Ipn{\mu_{t}^{(n)},f}
  &=\frac{1}{n}\sum_{i=1}^{n}f(\lambda_{i}^{(n)}(t))
	=\frac{1}{n}\sum_{i=1}^{n}f(\Phi_{i}^{n}(Y^{(n)}(t)))
	=: F_{n}(n^{-1/2}X^{(n)}(t)),
\end{align}
where $F_{n}(C)$, for $C\in\R^{\frac{n(n+1)}{2}}$, is such that
\begin{align*}
F_{n}(C)
  &:=\frac{1}{n}\sum_{i=1}^{n}f(\Phi_{i}^{n}(A^{(n)}+C)).
\end{align*}
Next we show that the right hand side of \eqref{eq:muf1} satisfies the conditions of Theorem \ref{thm:chainrule} for a suitable choice of $\mathcal{D}$ and $M$. 

Observe that $\Phi^{n}$ is infinitely differentiable in the set of symmetric matrices whose  characteristic polynomials do not have multiple roots, or equivalently, the matrices without multiple eigenvalues. As a consequence, the mapping $x\mapsto \Phi(x+A^{(n)})$ is differentiable in the complement of 
\begin{align}\label{eq:Mdef}
M_{A^{(n)}}:=\{x\in\R^{\frac{n(n+1)}{2}}\ |\ p(x+A^{(n)})=0\},
\end{align}
where $p:\R^{\frac{n(n+1)}{2}}\rightarrow\R$ denotes the discriminant of the matrix induced by $x$, defined by $p(x)=\prod_{i\neq j}(\Phi^n_{i}(x)-\Phi^n_{j}(x))^2$. It is well known that $p$ is a polynomial in the entries of $x$ (see \cite[Appendix~A.4]{AnGuZe} for a proof of this fact) and consequently, $M_{A^{(n)}}$ is an algebraic variety. Moreover, by a result by Von Neumann and Wigner (see \cite{Lax}), $M_{A^{(n)}}$ has codimension 2, namely, the maximal dimension of the tangent vector spaces at the non-singular points of $M_{A^{(n)}}$ is equal to $\frac{n(n+1)}{2}-2$. As a consequence, the projection $\gamma_{j}(M_{A^{(n)}})$ is a variety of codimension at least 1 embedded in $P_{j}:=\{(x_{1},\dots, x_{d})\ |\ x_{j}=0\},$ and thus $\gamma_{j}(M_{A^{(n)}})$ has Lebesgue measure zero on $P_{j}$. From here we conclude that condition (1) in Theorem \ref{thm:chainrule} holds for $\mathcal{D}:=\R^{\frac{n(n+1)}{2}}\backslash M_A^{(n)}$.

 Next we prove condition \eqref{cond1:thmchain}.  First we observe  for every $C\in G$ that
 the partial derivative of $F_{n}(n^{-1/2}C)$ with respect to the $(k,h)$-th component, denoted by $\frac{\partial }{\partial y_{k,h}}F_{n}(n^{-1/2}C)$, is given by
\begin{align}\label{eq:DFid}
\frac{\partial }{\partial y_{k,h}}F_{n}(n^{-1/2}C)
	&= \frac{1}{n^{\frac{3}{2}}}\sum_{i=1}^{n}f^{\prime}(\Phi_{i}^{n}(n^{-1/2}C+A^{(n)}))\frac{\partial \Phi_{i}^{n}}{\partial y_{k,h}}(n^{-1/2}C+A^{(n)}),
\end{align}
Hence, using the fact that $\big|\frac{\partial \Phi_{i}^{n}}{\partial y_{k,h}}\big|\leq 2$ (see equation \eqref{eq:D1PhiU}), we get
\begin{align*}
\Abs{\frac{\partial }{\partial y_{k,h}}F_{n}(n^{-1/2}C)}
	&\leq \frac{1}{n^{\frac{3}{2}}}\Norm{f^{\prime}}_{\infty}\sum_{i=1}^{n}\bigg{|}\frac{\partial \Phi_{i}^{n}}{\partial y_{k,h}}(n^{-1/2}C+A^{(n)})\bigg{|}
	\leq \frac{2}{\sqrt{n}}\Norm{f^{\prime}}_{\infty}.
\end{align*}
Using the previous inequality, we conclude that condition	\eqref{cond1:thmchain} holds.

 To prove condition \eqref{cond2:thmchain} in Theorem \ref{thm:chainrule}, we see that for every $1\leq k\leq h\leq n$ and $t>0$ fixed,
\begin{align}\label{eq:D2Fid}
\frac{\partial^2 }{\partial y_{k,h}^2}F_{n}(n^{-1/2}X^{(n)}(t))
  &=\frac{1}{n^2}\sum_{i=1}^{n}f^{\prime\prime}(\Phi_{i}^{n}(Y^{(n)}(t)))\left(\frac{\partial \Phi_{i}^{n}}{\partial y_{k,h}}(Y^{(n)}(t))\right)^{2}\nonumber\\
	&+\frac{1}{n^2}\sum_{i=1}^{n}f^{\prime}(\Phi_{i}^{n}(Y^{(n)}(t)))\frac{\partial^2 \Phi_{i}^{n}}{\partial y_{k,h}^2}(Y^{(n)}(t)),
\end{align}
and hence, using relations $\Abs{\frac{\partial \Phi_{i}^{n}}{\partial y_{k,h}}}\leq 2$ and $\Abs{\frac{\partial^{2}\Phi_{i}^{n}}{\partial y_{k,h}^2}}\leq\sum_{i\neq j}\frac{4}{\lambda^{(n)}_{i}-\lambda^{(n)}_{j}}$ (see equation \eqref{eq:D2PhiU}), we obtain 
\begin{align}\label{ineq:D2Flambda}
\E\bigg[\bigg|\frac{\partial^2 }{\partial y_{k,h}^2}F_{n}(n^{-1/2}X^{(n)}(t))\bigg|\bigg]
  &\leq\frac{4}{n}\Norm{f^{\prime\prime}}_{\infty}+\frac{1}{n^2}\sum_{i=1}^{n}\E\bigg[\bigg|\frac{\partial^2 \Phi_{i}^{n}}{\partial y_{k,h}^2}(Y^{(n)}(t))\bigg|\bigg]\nonumber\\
	&\leq\frac{4}{n}\Norm{f^{\prime\prime}}_{\infty}+\frac{4}{n^2}\sum_{i\neq j}\E\bigg[|\lambda_{i}^{(n)}(t)-\lambda_{j}^{(n)}(t)|^{-1}\bigg].
\end{align}
To show the right hand side is finite we proceed as follows. For $x\in\R^{\frac{n(n+1)}{2}}$, let $\phi_{\varepsilon}(x)$ denote the Gaussian kernel of variance $\varepsilon$. We can easily check that there exists constants $C>0$ and $\sigma>0$ only depending on $n, A^{(n)}$ and $R(t,t)$ such that, after identifying $A^{(n)}$ as an element of $\R^{\frac{d(d+1)}{2}}$, 
$$\phi_{R(t,t)^{\frac{1}{2}}}(n^{-1/2}(y-A^{(n)}))\leq C \phi_{\sigma}(x),$$
and consequently,
\begin{align}\label{eq:invmomentsphi}
\E\bigg[|\lambda_{i}^{(n)}(t)-\lambda_{j}^{(n)}(t)|^{-1}\bigg]
  &=\int_{\R^{\frac{n(n+2)}{2}}}|\Phi_{i}^n(n^{-1/2}x+A^{(n)})-\Phi_{j}^n(n^{-1/2}x+A^{(n)})|^{-1}\phi_{R(t,t)^{\frac{1}{2}}}(x)\ud x\nonumber\\
	&=\int_{\R^{\frac{n(n+2)}{2}}}|\Phi_{i}^n(y)-\Phi_{j}^n(y)|^{-1}\phi_{R(t,t)^{\frac{1}{2}}}(n^{-1/2}(y-A^{(n)}))\ud x\nonumber\\
	&\leq C\int_{\R^{\frac{n(n+2)}{2}}}|\Phi_{i}^n(x)-\Phi_{j}^n(x)|^{-1}\phi_{\sigma}(x)\ud x\nonumber\\
	&= C\sigma^{\frac{n(n+2)}{2}}\int_{\R^{\frac{n(n+2)}{2}}}|\Phi_{i}^n(x)-\Phi_{j}^n(x)|^{-1}\phi_{1}(x)\ud x.
\end{align}
Similarly to \cite[Equation~(5.6)]{NuPe}, we can use the joint density of the eigenvalues of a standard GOE of dimension $n$, to deduce that the right hand side of \eqref{eq:invmomentsphi} is finite. Hence, from \eqref{ineq:D2Flambda} we conclude that 
$$\E\bigg[\bigg|\frac{\partial^2 }{\partial y_{k,h}^2}F_{n}(n^{-1/2}X^{(n)}(t))\bigg|\bigg]<\infty.$$
Moreover, by \eqref{eq:D2Phisum} and \eqref{eq:D2Fid}, we have 
\begin{align*}
\sum_{k\leq h}\frac{\partial^2 }{\partial y_{k,h}^2}F_{n}(n^{-1/2}X^{(n)}(t))
  &=\frac{2}{n^2}\sum_{i=1}^{n}f^{\prime\prime}(\Phi_{i}^{n}(Y^{(n)}(t)))+\frac{2}{n^2}\sum_{i\neq j}\frac{f^{\prime}(\Phi_{i}^{n}(Y^{(n)}(t)))}{\Phi_{i}^{n}(Y^{(n)}(t))-\Phi_{j}^{n}(Y^{(n)}(t))}\\
	&\hspace{-1.4cm}=\frac{2}{n^2}\sum_{i=1}^{n}f^{\prime\prime}(\Phi_{i}^{n}(Y^{(n)}(t)))+\frac{1}{n^2}\sum_{i\neq j}\frac{f^{\prime}(\Phi_{i}^{n}(Y^{(n)}(t)))-f^{\prime}(\Phi_{j}^{n}(Y^{(n)}(t)))}{\Phi_{i}^{n}(Y^{(n)}(t))-\Phi_{j}^{n}(Y^{(n)}(t))},
\end{align*}
where we have used
\[
\sum_{i\neq j}\frac{f^{\prime}(\Phi_{i}^{n}(Y^{(n)}(t)))}{\Phi_{i}^{n}(Y^{(n)}(t))-\Phi_{j}^{n}(Y^{(n)}(t))}=-\sum_{i\neq j}\frac{f^{\prime}(\Phi_{j}^{n}(Y^{(n)}(t)))}{\Phi_{i}^{n}(Y^{(n)}(t))-\Phi_{j}^{n}(Y^{(n)}(t))}.
\]
Thus, by the mean value theorem, we conclude that
\begin{align*}
\bigg|\sum_{k\leq h}\frac{\partial^2 }{\partial y_{k,h}^2}F_{n}(n^{-1/2}X^{(n)}(t))\bigg|
  &\leq\frac{4}{n}\Norm{f^{\prime\prime}}_{\infty},
\end{align*}
which proves relation \eqref{cond2:thmchain}. Therefore, the right hand side of \eqref{eq:muf1} satisfies the conditions of Theorem \ref{thm:chainrule}. As a consequence, 
\begin{align}\label{eq:evolutionprevp}
\Ipn{\mu_{t}^{(n)},f}-\Ipn{\mu_{0}^{(n)},f}
  &=\sum_{1\leq k\leq h\leq n}\int_{0}^{t}\mathbf{1}_{\mathcal{D}}(X^{(n)}(s))\frac{\partial F_{n}}{\partial y_{k,h}}(n^{-1/2}X^{(n)}(s))\delta X_{k,h}(s)\nonumber\\
	&+\frac{1}{2}\sum_{1\leq k\leq h \leq n}\int_{0}^{t}\mathbf{1}_{\mathcal{D}}(X^{(n)}(s))\frac{\partial^2 F_{n}}{\partial y_{k,h}^2}(n^{-1/2}X^{(n)}(s))\frac{\ud}{\ud s}R(s,s)\ud s.
\end{align}
Moreover, by Remark (i) after Theorem  \ref{thm:chainrule}, the indicators $\mathbf{1}_{\mathcal{D}}(X^{(n)}(s))$ can be replaced by $\mathbf{1}_{{G}}(X^{(n)}(s))$, which leads to
\begin{align}\label{eq:evolutionprev}
\Ipn{\mu_{t}^{(n)},f}-\Ipn{\mu_{0}^{(n)},f}
  &=\sum_{1\leq k\leq h\leq n}\int_{0}^{t}\mathbf{1}_{{G}}(X^{(n)}(s))\frac{\partial F_{n}}{\partial y_{k,h}}(n^{-1/2}X^{(n)}(s))\delta X_{k,h}(s)\nonumber\\
	&+\frac{1}{2}\sum_{1\leq k\leq h \leq n}\int_{0}^{t}\mathbf{1}_{{G}}(X^{(n)}(s))\frac{\partial^2 F_{n}}{\partial y_{k,h}^2}(n^{-1/2}X^{(n)}(s))\frac{\ud}{\ud s}R(s,s)\ud s.
\end{align}

From relations \eqref{eq:D2Phisum} and \eqref{eq:D2Fid}, we deduce that
\begin{multline}\label{eq:D2Fid2}
\mathbf{1}_{{G}}(X^{(n)}(t))\sum_{1\leq k\leq h \leq n}\frac{\partial^2 }{\partial y_{k,h}^2}F_{n}(n^{-1/2}X^{(n)}(t))\\
\begin{aligned}
  &=\frac{\mathbf{1}_{{G}}(X^{(n)}(t))}{n^2}\sum_{i=1}^{n}f^{\prime\prime}(\Phi_{i}^{n}(Y^{(n)}(t)))\sum_{1\leq k\leq h \leq n}\left(\frac{\partial \Phi_{i}^{n}}{\partial y_{k,h}}(Y^{(n)}(t))\right)^{2}\\
	&+\frac{\mathbf{1}_{{G}}(X^{(n)}(t))}{n^2}\sum_{i=1}^{n}f^{\prime}(\Phi_{i}^{n}(Y^{(n)}(t)))\sum_{1\leq k\leq h \leq n}\frac{\partial^2 \Phi_{i}^{n}}{\partial y_{k,h}^2}(Y^{(n)}(t))\\
	&=\frac{2\mathbf{1}_{{G}}(X^{(n)}(t))}{n^2}\left(\sum_{i=1}^{n}f^{\prime\prime}(\Phi_{i}^{n}(Y^{(n)}(t)))
	+\sum_{i=1}^{n}\sum_{j\neq  i}\frac{f^{\prime}(\Phi_{i}^{n}(Y^{(n)}(t)))}{\lambda^{(n)}_{i}-\lambda^{(n)}_{j}}\right).
\end{aligned}
\end{multline}
Combining \eqref{eq:DFid}, \eqref{eq:evolutionprev} and \eqref{eq:D2Fid2}, we get 
\begin{align*}
\big\langle\mu_{t}^{(n)},f\big\rangle
  &=\big\langle\mu_{0}^{(n)},f\big\rangle +\frac{1}{n^{2}}\sum_{i=1}^{n}\int_{0}^{t}f^{\prime\prime}(\Phi_{i}^{n}(Y^{(n)}(s)))\frac{\ud}{\ud s}R(s,s)\ud s\\
  &+\frac{1}{n^{\frac{3}{2}}}\sum_{i=1}^{n}\sum_{k\leq h}\int_{0}^{t}f^{\prime}(\Phi_{i}^{n}(Y^{(n)}(s)))\mathbf{1}_{G}(Y^{(n)}(s))\frac{\partial \Phi_{i}^{n}}{\partial y_{k,l}}(Y^{(n)}(s))\delta X_{k,h}(s)\\
	&+\frac{1}{2}\int_{0}^{t}\int_{\R^2}\mathbf{1}_{\{x\neq y\}}\frac{f^{\prime}(x)-f^{\prime}(y)}{x-y}\frac{\ud}{\ud s}R(s,s)\mu_{s}^{(n)}(\ud x)\mu_{s}^{(n)}(\ud y)\ud s.
\end{align*}
Equation \eqref{eq:evolution} then follows from the fact that for every $s>0$,
\begin{align*}
\frac{1}{n^{2}}\sum_{i=1}^{n}f^{\prime\prime}(\Phi_{i}^{n}(Y^{(n)}(s)))
  &=\int_{\R^2}\mathbf{1}_{\{x=y\}}f^{\prime\prime}(x)\mu_{s}^{n}(dx)\mu_{s}^{n}(dy)
\end{align*}
.
\end{proof}
\section{Tightness of the family of laws $\{\mu^{(n)}, n\geq 1\}$.}\label{sec:tightness}
In order to prove tightness for the family $\{\mu^{(n)}, n\geq1\},$ we follow the approach presented in \cite{PaGaPe}. Namely, we show that for every test function $f$ belonging to the set $\Cc^1(\R)$, the process $\langle \mu_{t}^{(n)}, f\rangle$ satisfies the Billingsley criteria.
\begin{Proposition}\label{prop:tightness}
Assume that $R(s,t)$ satisfies hypothesis \textbf{(H2)}. Then, almost surely, the family of measures $\{\mu^{(n)}, n\geq 1\}$ is tight in the space $\mathcal{C}(\R_{+},\mathtt{Pr}(\R))$. 
\end{Proposition}
\begin{proof}
We follow the same argument as in \cite[Proposition~1]{PaGaPe}. It suffices to prove that for every bounded function $f\in \Cc^{1}(\R)$ with bounded derivative, the process $\{(\Ipm{\mu_{t}^{(n)},f},  t\geq0), n\ge 1\}$ is tight. To show this, we observe  that, since $\mu_{0}^{(n)}$ converges weakly, by  Billingsley's criteria (see \cite[Theorem~12.3]{Bi}), it is enough to show that there exist constants $C,p>0$ and $q>1$, independent of $n$, such that for every $0\leq t_{1}\leq t_{2}$, 
\begin{align}\label{ineq:Billingsleycondition}
\E\left[\Abs{\Ipn{\mu_{t_{1}}^{(n)},f}-\Ipn{\mu_{t_{2}}^{(n)},f}}^{p}\right]
  &\leq C\Abs{t_{2}-t_{1}}^{q}.
\end{align}
To prove \eqref{ineq:Billingsleycondition} we proceed as follows.  By the Cauchy-Schwarz inequality, 
\begin{align*}
\big|\Ipn{\mu_{t_{1}}^{(n)},f}-\Ipn{\mu_{t_{2}}^{(n)},f}\big|
  &=\left|\frac{1}{n}\sum_{i=1}^{n}f(\lambda_{i}^{(n)}(t_{2}))-f(\lambda_{i}^{(n)}(t_{1}))\right|\\
	&\leq \left|\frac{1}{n}\sum_{i=1}^{n}\Big(f(\lambda_{i}^{(n)}(t_{2}))-f(\lambda_{i}^{(n)}(t_{1}))\Big)^2\right|^{\frac{1}{2}}\\
	&\leq\Norm{f^{\prime}}_{\infty}\left|\sum_{i=1}^{n}\frac{1}{n}\Big(\lambda_{i}^{(n)}(t_{2})-\lambda_{i}^{(n)}(t_{1})\Big)^2\right|^{\frac{1}{2}}.
\end{align*}
\noindent Hence, using the Hoffman-Weilandt inequality (see \cite[Lemma~2.1.19]{AnGuZe}), as well as the symmetry of $Y^{(n)}(t)$, we deduce that for every $1\leq j\leq n$, 
\begin{align}\label{ineq:lambdadiff}
\big|\Ipn{\mu_{t_{1}}^{(n)},f}-\Ipn{\mu_{t_{2}}^{(n)},f}\big|
  &\leq \Norm{f^{\prime}}_{\infty}\left(\frac{1}{n}\sum_{i=1}^{n}\Big(\lambda_{i}^{(n)}(t_{2})-\lambda_{i}^{(n)}(t_{1})\Big)^2\right)^{\frac{1}{2}}\nonumber\\
	&\leq \Norm{f^{\prime}}_{\infty}\left(\frac{1}{n}Tr\Big(Y^{(n)}(t_{2})-Y^{(n)}(t_{1})\Big)^2\right)^{\frac{1}{2}}\nonumber\\
	&= \Norm{f^{\prime}}_{\infty}\left(\frac{1}{n}\sum_{i,k=1}^{n}\Big(Y_{i,k}^{(n)}(t_2)-Y_{i,k}^{(n)}(t_{1})\Big)^2\right)^\frac{1}{2}.
\end{align}
By condition \textbf{(H2)}, we have that  for all $\gamma>0$,
\[
\E\Big[\Big(Y_{i,k}^{(n)}(t_2)-Y_{i,k}^{(n)}(t_{1})\Big)^{2(\gamma+1)/\gamma}\Big]\leq2\frac{\kappa}{n}\Abs{t_{2}-t_{1}},
\]
for some constants $\kappa,\gamma>0$, and consequently, by \eqref{ineq:lambdadiff},
\begin{align}\label{ineq:lambdanorm}
\big\|\Ipn{\mu_{t_{1}}^{(n)},f}-\Ipn{\mu_{t_{2}}^{(n)},f}\big\|_{L^{\frac{2\gamma+2}{\gamma}}(\Omega)}
  &\leq\Norm{f^{\prime}}_{\infty}\E\left[\left(\frac{1}{n}\sum_{i,k=1}^{n}\left(Y_{i,k}^{(n)}(t_{2})-Y_{i,k}^{(n)}(t_{1})\right)^2\right)^\frac{\gamma+1}{\gamma}\right]^{\frac{\gamma}{2\gamma+2}}\nonumber\\
	&=\Norm{f^{\prime}}\left\|\frac{1}{n}\sum_{i,k=1}^{n}\left(Y_{i,k}^{(n)}(t_{2})-Y_{i,k}^{(n)}(t_{1})\right)^2\right\|_{L^{\frac{\gamma+1}{\gamma}}(\Omega)}^{\frac{1}{2}}\nonumber\\
  &\leq\Norm{f^{\prime}}\left(\frac{1}{n}\sum_{i,k=1}^{n}\Norm{\Big(Y_{i,k}^{(n)}(t_{2})-Y_{i,k}^{(n)}(t_{1})\Big)^2}_{L^{\frac{\gamma+1}{\gamma}}(\Omega)}\right)^{\frac{1}{2}}\nonumber\\
	&\leq C\Norm{f^{\prime}}_{\infty}\Abs{t_{2}-t_{1}}^{\frac{1}{2}},
\end{align}
for some universal constant $C>0$. The latter implies,
\begin{align*}
\E\left[\Abs{\Ipn{\mu_{t_{1}}^{(n)},f}-\Ipn{\mu_{t_{2}}^{(n)},f}}^{\frac{2\gamma+2}{\gamma}}\right]
  &\leq C\Norm{f^\prime}_{\infty}\Abs{t_{2}-t_{1}}^{1+\frac{1}{\gamma}}.
\end{align*}
Thus Billingsley's critera \eqref{ineq:Billingsleycondition} holds for $p=\frac{2\gamma+1}{\gamma}$ and $q=1+\frac{1}{\gamma}$. The proof is now complete.
\end{proof}
\section{Weak convergence of the empirical measure of eigenvalues}\label{sec:convergence}
This section is devoted to the proof of Theorem \ref{thm:main}. It is worth mentioning that, although some of the arguments we present are similar to \cite{PaGaPe}, our estimations are very different, mainly due to the fact that we do not have an analogue for Meyers' inequality for the extended Skorohod integral.

The following Proposition is useful for the proof of Theorem \ref{thm:main}. Its proof will be given at the end of this section.
\begin{Proposition}\label{Prop:aux}
For every $t>0$ fixed, the random variable 
\begin{align}\label{eq:Gdef}
G_{r}:=\frac{1}{n_{r}^{\frac{3}{2}}}\sum_{i=1}^{n_{r}}\sum_{k\leq h}\int_{0}^{t}f^{\prime}(\Phi_{i}^{n_{r}}(Y^{(n_{r})}(s)))\mathbf{1}_{G}(Y^{(n_{r})}(s))\frac{\partial \Phi_{i}^{n_{r}}}{\partial y_{k,l}}(Y^{(n_{r})}(s))\delta X_{k,h}(s),
\end{align}
converges to zero in $L^{2}(\Omega)$ as $n\to\infty$.
\end{Proposition}

\begin{proof}[Proof of Theorem \ref{thm:main}:]
By Lemma \ref{prop:tightness}, the sequence $\{\mu^{(n)}, n\geq1\}$ is tight, which implies that there exists a subsequence $\{\mu^{(n_r)}, r\geq1\}$ that converges in law, in the topology of $\mathcal{C}(\R_{+},\mathtt{Pr}(\R))$, to a measure valued stochastic process $\mu=(\mu_{t}, t\geq0)$. Then,  if we show that $\mu$ is deterministic, we conclude that $\{\mu^{(n)}, n\geq1\}$ converges in probability to $\mu$.

\noindent Using Proposition \ref{Prop:aux} together with relation \eqref{eq:evolution}, we deduce that the sequence of random variables
\begin{align}\label{aux_6}
\big\langle\mu_{t}^{(n_{r})},f\big\rangle
  -\big\langle\mu_{0}^{(n_r)},f\big\rangle
	-\frac{1}{2}\int_{0}^{t}\int_{\R^2}\frac{f^{\prime}(x)-f^{\prime}(y)}{x-y}\frac{\ud}{\ud s}R(s,s)\mu_{s}^{(n_{r})}(\ud x)\mu_{s}^{(n_r)}(\ud y)\ud s\notag\\
	-\frac{1}{2n_{r}^{2}}\sum_{i=1}^{n_{r}}\int_{0}^{t}f^{\prime\prime}(\Phi_{i}^{n_{r}}(Y^{(n_{r})}(s)))\frac{\ud}{\ud s}R(s,s)\ud s,
\end{align}
converges to zero in $L^{2}(\Omega)$. In particular, since  $\mu^{(n_{r})}$ converges in law to $\mu$, it implies  that $\mu$ satisfies the following measure-valued differential equation
\begin{align}\label{eq:dynamicsmu}
\Ipm{\mu_{t},f}
  &=\Ip{\mu_0,f}+\frac{1}{2}\int_{0}^{t}\int_{\R^{2}}\frac{f^{\prime}(x)-f^{\prime}(y)}{x-y}\frac{\ud}{\ud s}R(s,s)\mu_{s}(\ud x)\mu_{s}(\ud y)\ud s,
\end{align}
for each $t\geq0$ and $f\in \mathcal{C}_{b}^{2}(\R)$. Then we can conclude that any weak limit of a subsequence  $\{\mu^{(n_r)}, r\geq1\}$ should satisfy \eqref{eq:dynamicsmu}. 
We now proceed to prove that  $\mu$ is characterized by \eqref{eq:dynamicsmup}. In order to do so, we apply  \eqref{eq:dynamicsmu} to the sequence of functions 
$$f_{z}(x)
  =\frac{1}{x-z},\ \ \ \ \ z\in\Q^2\cap\C_{+},$$
we get 
\begin{align*}
\Ipm{\mu_{t},f_z}
  &=\Ip{\mu_0,f_z}+\frac{1}{2}\int_{0}^{t}\int_{\R^{2}}\frac{(x-z)+(y-z)}{(x-z)^2(y-z)^2}\frac{\ud}{\ud s}R(s,s)\mu_{s}(\ud x)\mu_{s}(\ud y)\ud s\nonumber\\
	&=\Ip{\mu_0,f_z}+\int_{0}^{t}\int_{\R^{2}}\frac{1}{(x-z)(y-z)^2}\frac{\ud}{\ud s}R(s,s)\mu_{s}(\ud x)\mu_{s}(\ud y)\ud s,
\end{align*}
where the last identity follows from the symmetry over the variables $x$ and $y$. 
Therefore, using a continuity argument, we get that the Cauchy-Stieltjes transform $G_{t}(z):=\int_\R\frac{1}{x-z}\mu_{t}(dz)$, defined in the domain $\C^{+}$, satisfies the integral equation 
\begin{align*}
G_{t}(z)
	&=\Ip{\mu_0,f}+\int_{0}^{t}\int_{\R^{2}}\frac{1}{(x-z)(y-z)^2}\frac{ \ud}{\ud s}R(s,s)\mu_{s}(\ud x)\mu_{s}(\ud y)\ud s\\
  &=\Ip{\mu_0,f}+\int_{0}^{t}\frac{\ud}{\ud s}R(s,s)G_{s}(z)\frac{\partial}{\partial z}G_{s}(z)\ud s.
\end{align*} 
In particular,
\begin{align*}
G_{t}(z)
  &=F_{R(t,t)}(z),
\end{align*}
where $F_{\tau}(z)$, for $z\in\C_{+}$, is the unique solution to the Burgers' equation
\begin{align*}
\frac{\partial}{\partial \tau}F_{\tau}(z)
  &=F_{\tau}(z)\frac{\partial}{\partial z}F_{\tau}(z),\nonumber\\
F_{0}(z)
  &=\Ip{\mu_0,f},
\end{align*}
which completes the proof.
\end{proof}

Finally we prove Proposition \ref{Prop:aux}.
\begin{proof}[Proof of Proposition \ref{Prop:aux}]
By relation \eqref{eq:evolution}, we have that 
\begin{align}\label{eq:Gdual}
G_{r}
  &=\Ipm{\mu_{t}^{(n_{r})},f}-\Ipm{\mu_{0}^{(n_{r})},f}-\frac{1}{2n_r^{2}}\sum_{i=1}^{n_r}\int_{0}^{t}f^{\prime\prime}(\Phi_{i}^{n}(Y^{(n_r)}(s)))\frac{\ud}{\ud s}R(s,s)\ud s\nonumber\\
	&\hspace{2cm}-\frac{1}{2}\int_{0}^{t}\int_{\R^2}\frac{f^{\prime}(x)-f^{\prime}(y)}{x-y}\frac{\ud}{\ud s}R(s,s)\mu_{s}^{(n_r)}(\ud x)\mu_{s}^{(n_r)}(\ud y)\ud s,
\end{align}
and consequently, we can write
\begin{align}\label{eq:EGdual}
\E\left[G_{r}^{2}\right]
  &=\E\left[(\langle\mu_{t}^{(n_{r})},f\rangle-\langle\mu_{0}^{(n_{r})},f\rangle)G_{r}\right]\\
  &-\frac{1}{2n_r^{2}}\sum_{i=1}^{n_r}\int_{0}^{t}
	\E\left[f^{\prime\prime}(\Phi_{i}^{n_r}(Y^{(n_{r})}(s)))G_{r}\right]\frac{\ud}{\ud s}R(s,s)\ud s\nonumber\\
	&-\frac{1}{2}\int_{0}^{t}\E\left[\int_{\R^2}\frac{f^{\prime}(x)-f^{\prime}(y)}{x-y}\frac{\ud}{\ud s}R(s,s)\mu_{s}^{(n_r)}(\ud x)\mu_{s}^{(n_r)}(\ud y)G_{r}\right]\ud s.\nonumber
\end{align}
Next we bound the terms appearing in the right hand side. Using relation \eqref{eq:Gdual}, as well as the fact that $f^{\prime}$ and $f^{\prime\prime}$ are bounded, we can easily show that for every $T>0$, there exists a constant $C>0$, only depending on $T$ and the properties of $R(s,t)$, such that for every $t\in[0,T]$,
\begin{align}\label{ineq:Grbounded}
\Abs{G_{r}}\leq C(\Norm{f}_{\infty}+\Norm{f^{\prime\prime}}_{\infty}).
\end{align}
From here we obtain
\begin{multline*}
\Abs{\frac{1}{n_r^{2}}\sum_{i=1}^{n_r}\int_{0}^{t}
	\E\left[f^{\prime\prime}(\Phi_{i}^{n_r}(Y^{(n_{r})}(s)))G_{r}\right]\frac{\ud}{\ud s}R(s,s)\ud s}\\
	  \hspace{-2cm}\leq \frac{C}{n_r}\Norm{f^{\prime\prime}}_{\infty}(\Norm{f}_{\infty}+\Norm{f^{\prime\prime}}_{\infty})\int_{0}^{T}\Abs{R(s,s)}\ud s,
\end{multline*}
and hence 
\begin{align}\label{conv:T1}
\lim_{r\rightarrow\infty}\Abs{\frac{1}{n_r^{2}}\sum_{i=1}^{n_r}\int_{0}^{t}
	\E\left[f^{\prime\prime}(\Phi_{i}^{n_r}(Y^{(n_{r})}(s)))G_{r}\right]\frac{\ud}{\ud s}R(s,s)\ud s}=0.
\end{align}
Next we notice, by the zero mean property of $G_{r}$,  that
\begin{align*}
\E\left[\Big(\langle\mu_{t}^{(n_{r})},f\rangle-\langle\mu_{0}^{(n_{r})},f\rangle \Big)G_{r}\right]
  =\E\left[\langle\mu_{t}^{(n_{r})},f\rangle G_{r}\right]
	=\frac{1}{n_{r}}\sum_{i=1}^{n_{r}}\E\left[f(\Phi_{i}(Y^{(n_{r})}(t)))G_{r}\right]. 
\end{align*}
Consequently,  using \eqref{eq:Gdef} and the duality property \eqref{eq:Skorohodduality},  we get 
\begin{multline}\label{eq:mufGr}
\E\Big[\Big(\langle\mu_{t}^{(n_{r})},f\rangle-\langle\mu_{0}^{(n_{r})},f\rangle\Big)G_{r}\Big]\\
  \begin{aligned}
	&=\frac{1}{n_{r}^{\frac{5}{2}}}\sum_{i,j=1}^{n_{r}}\E\Bigg[\Bigg\langle Df(\Phi_{i}^{n_{r}}(Y^{(n_{r})}(t))), \sum_{k\leq h}f^{\prime}(\Phi_{j}^{n_{r}}(Y^{(n_{r})}(s)))\\
	&\hspace{5cm}\times\mathbf{1}_{G}(Y^{(n_{r})}(s))\frac{\partial \Phi_{j}^{n_{r}}}{\partial y_{k,l}}(Y^{(n_{r})}(s))1_{[0,t]}(s)\Bigg\rangle\Bigg]\\
	&=\frac{1}{n_{r}^{3}}\sum_{i,j=1}^{n_{r}}\sum_{k\leq h}\E\bigg[f^{\prime}(\Phi_{i}^{n_{r}}(Y^{(n_{r})}(t)))\frac{\partial \Phi_{j}^{n_{r}}}{\partial y_{k,l}}(Y^{(n_{r})}(t))\\
	&\hspace{1cm}\times \int_{0}^{t}f^{\prime}(\Phi_{j}^{n_{r}}(Y^{(n_{r})}(s)))\mathbf{1}_{G}(Y^{(n_{r})}(s))\frac{\partial \Phi_{j}^{n_{r}}}{\partial y_{k,l}}(Y^{(n_{r})}(s))\frac{\partial R}{\partial s}(s,t)\ud s\bigg].
  \end{aligned}
\end{multline}
  On the other hand, by the Cauchy-Schwarz inequality and the relation \eqref{eq:D2Phisum}, we have that for every $u,v>0$ and $i,j\in \N$,
\begin{multline}\label{eq:cauchyphi}
\Abs{\sum_{k\leq h}\left(\frac{\partial \Phi_{i}^{n_{r}}}{\partial y_{k,l}}(Y^{(n_{r})}(u))\right)\left(\frac{\partial \Phi_{j}^{n_{r}}}{\partial y_{k,l}}(Y^{(n_{r})}(v))\right)}\\
\leq \left(\sum_{k\leq h}\left(\frac{\partial \Phi_{i}^{n_{r}}}{\partial y_{k,l}}(Y^{(n_{r})}(u))\right)^2\right)^{\frac{1}{2}}\left(\sum_{k\leq h}\left(\frac{\partial \Phi_{j}^{n_{r}}}{\partial y_{k,l}}(Y^{(n_{r})}(v))\right)^2\right)^{\frac{1}{2}}=2. 
\end{multline}
Hence, from  \eqref{eq:mufGr} we conclude that 
\begin{align*}
\E\left[\Big(\langle\mu_{t}^{(n_{r})},f\rangle-\langle\mu_{0}^{(n_{r})},f\rangle\Big)G_{r}\right]
	&\leq\frac{2\Norm{f^{\prime}}_{\infty}^2}{n_{r}}\int_{0}^{T}\Abs{\frac{\partial R}{\partial s}(s,t)}\ud s,
\end{align*}
and consequently, 
\begin{align}\label{conv:T2}
\lim_{r\rightarrow\infty}\E\left[\Big(\langle\mu_{t}^{(n_{r})},f\rangle-\langle\mu_{0}^{(n_{r})},f\rangle\Big)G_{r}\right]=0.
\end{align}
 Finally, we handle the third term in \eqref{eq:EGdual}. Using the following identity 
\begin{align*}
\frac{f^{\prime}(x)-f^{\prime}(y)}{x-y}
  &=\int_{0}^1f^{\prime\prime}(\theta x+(1-\theta)y)\ud \theta,
\end{align*}
we deduce that for every $s>0$,
\begin{align}\label{eq:fprimequotiont}
\E\left[\int_{\R^2}\frac{f^{\prime}(x)-f^{\prime}(y)}{x-y}\mu_{s}^{(n_r)}(\ud x)\mu_{s}^{(n_r)}(\ud y)G_{r}\right]
   &=  \frac{1}{n_{r}^2}\sum_{i,j=1}^{n}\int_{0}^{1}\E\left[f^{\prime\prime}(I_{i,j}^{s,r}(\theta))G_{r}\right]\ud \theta,
\end{align}
where 
\begin{align*}
I_{i,j}^{s,r}(\theta)
  &:=\theta\Phi_{i}^{n_{r}}(Y^{(n_{r})}(s))+(1-\theta)\Phi_{j}^{n_{r}}(Y^{(n_{r})}(s)).
\end{align*}
The term in the right hand side of \eqref{eq:fprimequotiont} can be estimated as follows. Define the processes 
\begin{align*}
\Lambda_{k,h}^{l,r}(u)
  &:=\frac{\partial \Phi_{l}^{n_{r}}}{\partial y_{k,h}}(Y^{(n_{r})}(u)).
\end{align*}
We can easily show that 
\begin{align*}
Df^{\prime\prime}(I_{i,j}^{s,r}(\theta))
  &=\frac{\theta}{\sqrt{n}}f^{\prime\prime\prime}(I_{i,j}^{s,r}(\theta))\Lambda_{k,h}^{i,r}(s)\mathbf{1}_{[0,s]}+\frac{1-\theta}{\sqrt{n}}f^{(\prime\prime\prime)}(I_{i,j}^{s,r}(\theta))\Lambda_{k,h}^{j,r}(s)\mathbf{1}_{[0,s]}.
\end{align*}
Then, using the duality relation of the Skorohod integral, as well as the expression \eqref{eq:Gdef}, we obtain
\begin{multline*}
\E\left[f^{\prime\prime}(I_{i,j}^{s,r}(\theta))G_{r}\right]\\
\begin{aligned}
  &=\frac{1}{n_{r}^{2}}\sum_{l=1}^{n_{r}}\sum_{k\leq h}\E\bigg[\int_{0}^{t}\theta f^{\prime\prime\prime}(I_{i,j}^{s,r}(\theta))f^{\prime}(\Phi_{i}^{n_{r}}(Y^{(n_{r})}(s)))\Lambda_{k,h}^{i,r}(s)\Lambda_{k,h}^{l,r}(u)\frac{\partial R}{\partial u}(u,s)\textup{d}u\bigg]\\
	&+\frac{1}{n_{r}^{2}}\sum_{l=1}^{n_{r}}\sum_{k\leq h}\E\bigg[\int_{0}^{t}(1-\theta)f^{\prime\prime\prime}(I_{i,j}^{s,r}(\theta))f^{\prime}(\Phi_{i}^{n_{r}}(Y^{(n_{r})}(s)))\Lambda_{k,h}^{j,r}(s)\Lambda_{k,h}^{l,r}(u)\frac{\partial R}{\partial u}(u,s)\textup{d}u\bigg],
\end{aligned}
\end{multline*}
which, by the boundedness of $f^{\prime\prime\prime}$ and $f^{\prime}$, implies that there exists a constant $C>0$, only depending on $f$, such that
\begin{align}\label{eq:firstboundintf3}
\Abs{\E\left[f^{\prime\prime}(I_{i,j}^{s,r}(\theta))G_{r}\right]}
  &\leq \frac{C}{n_{r}^{2}}\E\bigg[\int_{0}^{t}\sum_{l=1}^{n_{r}}\bigg|\sum_{k\leq h}\Lambda_{k,h}^{i,r}(s)\Lambda_{k,h}^{l,r}(u)\bigg|\Abs{\frac{\partial R}{\partial u}(u,s)}\ud u\bigg]\nonumber\\
	&+\frac{C}{n_{r}^{2}}\E\bigg[\int_{0}^{t}\sum_{l=1}^{n_{r}}\bigg|\sum_{k\leq h}\Lambda_{k,h}^{j,r}(s)\Lambda_{k,h}^{l,r}(u)\bigg|\Abs{\frac{\partial R}{\partial u}(u,s)}\ud u\bigg].
\end{align}
Using \eqref{eq:cauchyphi} and \eqref{eq:firstboundintf3}, we get
\begin{align*}
\Abs{\E\left[f^{\prime\prime}(I_{i,j}^{s,r}(\theta))G_{r}\right]}
  &\leq \frac{4C}{n_{r}}\int_{0}^{t}\Abs{\frac{\partial R}{\partial u}(u,s)}\ud u
	\leq \frac{4Ct^{1-\frac{1}{\alpha}}}{n_{r}}\bigg(\int_{0}^{t}\Abs{\frac{\partial R}{\partial u}(u,s)}^{\alpha}\ud u\bigg)^{\frac{1}{\alpha}}\notag\\
	&\leq \frac{4Ct^{1-\frac{1}{\alpha}}}{n_{r}}\sup_{s\in[0,t]}\bigg(\int_{0}^{t}\Abs{\frac{\partial R}{\partial u}(u,s)}^{\alpha}\ud u\bigg)^{\frac{1}{\alpha}}.
\end{align*}
Using the previous identity in \eqref{eq:fprimequotiont}, we deduce that there exists a constant $C_1>0$, such that 
\begin{align}\label{conv:T3}
\Abs{\E\left[\int_{\R^2}\mathbf{1}_{\{x\neq y\}}\frac{f^{\prime}(x)-f^{\prime}(y)}{x-y}\mu_{s}^{(n_r)}(\ud x)\mu_{s}^{(n_r)}(\ud y)G_{r}\right]}
  &\leq \frac{C_1}{n_r}.
\end{align}
From \eqref{eq:EGdual}, \eqref{conv:T1}, \eqref{conv:T2} and \eqref{conv:T3}, we conclude that $G_{r}$ converges to zero in $L^{2}(\Omega)$,  as required.
\end{proof}
\noindent\textbf{Acknowledgements:} We would like to express our sincere gratitude   to Prof. David Nualart for his helpful observations and suggestions on the first draft of this paper. This research was supported by the Royal Society and CONACYT-Mexico.

\end{document}